\documentclass[a4text]{amsart}
\usepackage{a4wide}

\usepackage{amssymb}
\usepackage{amsthm}  
\usepackage{amsmath} 
\usepackage{amscd} 
\usepackage{url}
\usepackage{multirow}
\usepackage{hyperref}
\hypersetup{linktocpage}
\usepackage{color}
\usepackage{nicematrix}
\usepackage{amsaddr}
\usepackage{multicol}

\newcommand{\fk}{\mathfrak k}

\newcommand{\fh}{\mathfrak h}

\newcommand{\fc}{\mathfrak c}

\newcommand{\Ad}{{\rm Ad}}

\newcommand{\ad}{{\rm ad}}

\newtheorem{thm}{Theorem}[section]
\newtheorem{prop}{Proposition}[section]
\newtheorem*{prop*}{Proposition}
\newtheorem*{thm*}{Theorem}
\newtheorem{lem}{Lemma}[section]
\newtheorem*{lemma*}{Lemma}

\newtheorem*{cor*}{Corollary}

\newtheorem*{rem*}{Remark}
\theoremstyle{definition}
\newtheorem{def*}{Definition}[section]

\theoremstyle{remark}

\begin{document}
\title{Irreducible Killing and conformal Killing tensors on homogeneous plane waves}

\author{Jan Gregorovi\v c$^1$ and Lenka Zalabov\' a$^2$}
%\author{Jan Gregorovi\v{c}}
\address{
	$^1$Department of Mathematics, Faculty of Science, University of Ostrava, 701 03 Ostrava, Czech Republic, and Institute of Discrete Mathematics and Geometry, TU Vienna, Wiedner Hauptstrasse 8-10/104, 1040 Vienna, Austria. \\ ORCID ID: 0000-0002-0715-7911}\email{ jan.gregorovic@seznam.cz}

 %   \author{Lenka Zalabov\' a}
\address{$^2$Institute of Mathematics, 
Faculty of Science, University of South Bohemia, Brani\v sovsk\' a 1760, 370 05 \v Cesk\' e Bud\v ejovice, Czech Republic, and Department of Mathematics and Statistics, 
 	Faculty of Science, Masaryk University,	Kotl\'{a}\v{r}sk\'{a} 2, 611 37 Brno, 
Czech Republic. \\ ORCID ID: 0000-0002-3612-7068 }
\email{lzalabova@gmail.com}
    
\thanks{J.G. was supported by Austrian Science Fund (FWF): P34369; L.Z. is supported by the grant GACR 24-10887S Cartan supergeometries and Higher Cartan geometries.}

\begin{abstract}
This paper presents a classification of irreducible Killing and conformal Killing 2-tensors on homogeneous plane waves, a specific class of Lorentzian metrics on four-dimensional manifolds. Using the framework of BGG operators, we derive explicit formulae for these tensors and identify the conditions under which they exist.
\end{abstract}
\maketitle
%\tableofcontents

\section{Introduction}

In this paper, we consider a class of Lorentzian metrics $g$, homogeneous plane waves, on four-dimensional manifolds $M$. 
Homogeneous plane waves form a special class of Lorentzian spaces with a parallel null vector field, \cite{ML}. They play a crucial role in studying Lorentzian holonomy, \cite{hol}. In four dimensions, they serve as exact solutions to Einstein's field equations, \cite{stephani}.
It was recently shown in \cite{AG} that homogeneous plane waves are the only Lorentzian metrics that admit transitive essential conformal groups. 

The homogeneous plane waves, according to \cite{ML}, form two $3$-parameter classes of space-time metrics parametrized by $a_1,a_2,\gamma\in \mathbb{R}$ that in Brinkmann coordinates $(x_+,z_1,z_2,x_-)$ take the form
\begin{align*}
g_{0,a_1,a_2,\gamma}&:=2dx_+dx_-+dz_1^2+dz_2^2
+[z_1,z_2]\exp({\left[ \begin{smallmatrix}0 & -\gamma x_+\\ \gamma x_+& 0 \end{smallmatrix}\right]})\left[ \begin{smallmatrix}a_1 & 0\\ 0& a_2 \end{smallmatrix}\right]\exp({\left[ \begin{smallmatrix}0 & \gamma x_+\\- \gamma x_+ & 0 \end{smallmatrix}\right]})\left[ \begin{smallmatrix}z_1 \\z_2 \end{smallmatrix}\right]dx_+^2,\\
g_{1,a_1,a_2,\gamma}&:=2dx_+dx_-+dz_1^2+dz_2^2\\
&\: \: +[z_1,z_2]\exp({\left[ \begin{smallmatrix}0 & -\gamma\ln(x_+)\\ \gamma\ln(x_+)& 0 \end{smallmatrix}\right]})\left[ \begin{smallmatrix}a_1 & 0\\ 0& a_2 \end{smallmatrix}\right]\exp({\left[ \begin{smallmatrix}0 & \gamma\ln(x_+)\\- \gamma\ln(x_+)& 0 \end{smallmatrix}\right]})\left[ \begin{smallmatrix}z_1 \\z_2 \end{smallmatrix}\right]x_+^{-2}{dx_+^2}.
\end{align*}
Moreover, metrics $$g_{1,a_1,a_2,\gamma} \simeq g_{0,a_1+\frac14,a_2+\frac14,\gamma}$$ are conformally equivalent,  \cite{HS}, and metrics $$g_{0,a_1,a_2,\gamma} \simeq g_{0,\lambda^2a_1,\lambda^2a_2,\lambda\gamma} \simeq g_{0,a_2,a_1,\gamma}$$ are isometric, \cite{ML}. From a global viewpoint, homogeneous plane waves are investigated in \cite{HMZ}.

This paper aims to answer the question of which of these homogeneous plane waves carry irreducible Killing and conformal Killing tensors, respectively, \cite[Section 35.3]{stephani}.  Let us recall that Killing vectors are solutions to the equation 
$$\nabla_{(a}k_{b)}=0$$
acting on vectors $k_b$; here $g$ provides isomorphism between vectors and forms, and the parentheses around indexes mean symmetrization. These are the infinitesimal isometries of the homogeneous plane waves. 
Conformal Killing vectors are solutions to the equation 
 $$\nabla_{(a}k_{b)_0}=0$$
acting on vectors $k_b$; the subscript ${}_{0}$ indicates that we take the trace-free part. These are the infinitesimal conformal symmetries of the homogeneous plane waves.

Generalizing these equations for symmetric tensors $k_{bc}=k_{(bc)}$, one obtains the so-called Killing and conformal Killing ($2$-)tensors, sometimes referred to as hidden symmetries. There are also higher-order (conformal) Killing tensors, \cite[Section 35.3]{stephani}, which we will not study in this article and thus omit the prefix $2$. 
More precisely, Killing tensors are solutions to the equation
$$\nabla_{(a}k_{bc)}=0,$$
while conformal Killing tensors are trace-free solutions, i.e. satisfying $k_{bc}=k_{(bc)_0}$, to the equation
$$\nabla_{(a}k_{bc)_0}=0.$$

Two Killing vectors $k_a,k^\prime_a$ can be combined into a Killing tensor $k_{(a}k^\prime{}_{b)}$ and the metric itself is also a Killing tensor. These are called reducible Killing tensors, and a Killing tensor is irreducible if it cannot be written as a linear combination of reducible Killing tensors. Similarly, two conformal Killing tensors $k_a,k^\prime_a$ can be combined into a conformal Killing tensor $k_{(a}k^\prime{}_{b)_0}$. These are called reducible conformal Killing tensors, and a conformal Killing tensor is irreducible if it cannot be written as a linear combination of reducible conformal Killing tensors.

To study the tensors in question, we employ the fact that they are in the kernel of suitable overdetermined differential operators.
Moreover, these operators belong to the family of the so-called BGG operators depending on either

\begin{enumerate}
\item  the projective class $[\nabla]$ of the Levi-Civita connection $\nabla$ of $g$ in the case of Killing vectors and tensors, \cite{GL} (projective BGG operators), or
\item  the conformal class $[g]$ of $g$  in the case of conformal Killing vectors and tensors, \cite{Sem} (conformal BGG operators).
\end{enumerate}

We developed the theory for BGG operators on projective and conformal homogeneous manifolds in \cite{bgg-my} and \cite{bgg-conf}. In fact, we have shown that the solutions can be computed algebraically for particular homogeneous plane waves; we review the details in Section \ref{sec-bgg-how}. However, the parameters present new challenges that we must overcome in this paper. In particular, we can compute generic spaces of Killing and conformal Killing tensors. Still, in Propositions \ref{generic-CKT},  \ref{generic-fKT} and  \ref{generic-KT}, we show that these are reducible in most cases. Therefore, we need to investigate the algebraic subvarieties with non-generic equations where additional solutions can appear. These are determined by overdetermined systems of equations for the parameters we need to solve. For the conformal Killing tensors, we provide the complete analysis in Section \ref{sec-CKT}. We give explicit formulae for the irreducible conformal Killing tensors in Proposition \ref{expicit-CKT}. In particular, we prove the following theorem in Section \ref{sec-CKT}.

\begin{thm}\label{main-CKT}
\begin{enumerate}
\item The homogeneous plane waves $g_{0,\pm\frac{8}{3},\pm\frac{2}{3},0}$ have nine-dimensional spaces of irreducible conformal Killing tensors. 
\item The homogeneous plane wave $g_{0,-1,3 ,1}$ has a two-dimensional space of irreducible conformal Killing tensors.
\item The other homogeneous plane waves $g_{0,a_1,a_2, \gamma}$ and 
$g_{1,a_1,a_2, \gamma}$ are either (conformally) isometric to the cases (1) or (2) or do not possess any irreducible conformal Killing tensors.
\end{enumerate}
\end{thm}

 For Killing tensors on conformally flat homogeneous plane waves, we provide the complete analysis in Section \ref{sec-fKT}. We give explicit formulae for the irreducible conformal Killing tensors in Proposition \ref{expl-fKT}. These were previously investigated by different methods in \cite{KO} and we can confirm their results (including the case that was discarded in \cite[Proof of Theorem 2]{KO}). In particular, we prove the following theorem in Section \ref{sec-fKT}.

\begin{thm} \label{main-fKT}
\begin{enumerate}
\item
The homogeneous plane waves $ g_{0,\pm 1,\pm 1,0}$ and $g_{1,\frac34,\frac34,0}$ have one-dimensional spaces of irreducible Killing tensors. 

\item The homogeneous plane wave $g_{1,-\frac3{16},-\frac3{16},0}$ has a six-dimensional space of irreducible Killing tensors.

\item The other homogeneous plane waves $g_{0,a_1,a_2, \gamma}$ and $g_{1,a_1,a_2, \gamma}$ are either isometric to the cases (1) or (2) or do not possess any irreducible Killing tensors.
\end{enumerate}
\end{thm}

For the Killing tensors on conformally non-flat homogeneous plane waves, we provide the analysis in Section \ref{sec-KT}.
We find the explicit formulae for the irreducible conformal Killing tensors in Proposition \ref{expl-KT}. In particular, we prove the following theorem in Section \ref{sec-KT}.

\begin{thm}\label{main-KT}
\begin{enumerate}
    \item 
The homogeneous plane waves $g_{0,a_1,a_2,\gamma}$, $g_{1,0,a_2,0}$ and $g_{1,4a_2+\frac34,a_2,0}$ have one-dimensional spaces of irreducible Killing tensors.

\item The homogeneous plane waves $g_{0,0,\pm 2,0}$ and $g_{0,\pm \frac83,\pm \frac23,0}$ have two-dimensional spaces of irreducible Killing tensors.

\item The homogeneous plane wave $g_{1,0,\frac34,0}$ has a five-dimensional space of irreducible Killing tensors.

\item The homogeneous plane wave $g_{1,0,-\frac3{16},0}$ has a six-dimensional space of irreducible Killing tensors.

\item The other homogeneous plane waves $g_{0,a_1,a_2, \gamma}$ and $g_{1,a_1,a_2, \gamma}$ are either isometric to the above cases or do not possess any irreducible Killing tensors.
\end{enumerate}
\end{thm}

Finally, we compare Theorems \ref{main-KT} and \ref{main-CKT} and determine the irreducible Killing tensors that correspond to the irreducible conformal Killing tensors.

\begin{thm}\label{main-KTCKT}
\begin{enumerate}
    \item 
The homogeneous plane waves $g_{1,4a_2+\frac34,a_2,0}$ have one-dimensional spaces of irreducible Killing tensors corresponding to irreducible conformal Killing tensors.

\item The homogeneous plane waves $g_{0,\pm \frac83,\pm \frac23,0}$ have one-dimensional spaces of irreducible Killing tensors corresponding to irreducible conformal Killing tensors.

\item The homogeneous plane wave $g_{1,0,\frac34,0}$ has a four-dimensional space of irreducible Killing tensors corresponding to irreducible conformal Killing tensors.

\item The homogeneous plane wave $g_{1,0,-\frac3{16},0}$ has a five-dimensional space of irreducible Killing tensors corresponding to irreducible conformal Killing tensors.

\item The other homogeneous plane waves $g_{0,a_1,a_2, \gamma}$ and $g_{1,a_1,a_2, \gamma}$ are either isometric to the above cases or do not possess any irreducible Killing tensors corresponding to irreducible conformal Killing tensors.
\end{enumerate}
\end{thm}

\section{
Homogeneous plane waves in exponential coordinates
}

For the computations in this paper, we view the homogeneous plane waves as left-invariant metrics on the homogeneous space $K/H$ with Lie algebras $\fk=\langle k^1,k^2,k^3,k^4,k^5,k^6 \rangle, \fh=\langle k^5,k^6\rangle$, where we consider the Lie brackets (minus bracket of the Killing fields) 
\begin{gather*}
[k^1, k^2] = \gamma k^3+a_1k^5,\ \  [k^1, k^3] =-\gamma k^2+a_2k^6,\ \  [k^1, k^4] = -\epsilon k^4,\\
[k^1,k^5]=k^2-\epsilon k^5+\gamma k^6,\ \  [k^1,k^6]=k^3-\gamma k^5-\epsilon k^6,\ \ [k^2, k^5] =-k^4, [k^3,k^6]=-k^4,
\end{gather*}
found in \cite{ML,HMZ}. In the conformally flat case, where $a_1=a_2$, there is an additional Killing field that acts as a rotation on the planes $\langle k^2,k^3\rangle $ and $\langle k^5,k^6 \rangle$.

This description naturally suggests exponential coordinates for explicit computations. We consider the complement $\fc=\langle k^1,k^2,k^3,k^4 \rangle$ of $\fh$ in $\fk$ and on $K/H$, and we employ the exponential coordinates $$(x_1,x_2,x_3,x_4)\mapsto \exp(x_1k^1)\exp(x_2k^2+x_3k^3)\exp(x_4k^4)H.$$
As $ k^2,k^3,k^4,k^5,k^6 $ generate a Heisenberg Lie algebra, which is nilpotent, there is no problem with the exponentials as we can use the Baker-Campbell-Hausdorff formula. On the other hand, the adjoint action
\[ \Ad(\exp(x_1k^1))=\Ad \big(\exp ( \left[ \begin {smallmatrix} 0&0&0&0&0&0\\ 0&0&{
 -x_1}{ \gamma}&0&{ x_1}&0\\ 0&{ x_1}{
 \gamma}&0&0&0&{ x_1}\\ 0&0&0&{- x_1}\epsilon&0
&0\\ 0&{ x_1}{ a_1}&0&0&{ -x_1}\epsilon&{-
 x_1}{ \gamma}\\ 0&0&{ x_1}{ a_2}&0&{
 x_1}{ \gamma}&{ -x_1}\epsilon\end {smallmatrix} \right] 
) \big),\]
cannot be easily expressed for arbitrary values of the parameters. In particular, there are no simple coordinate formulae for the Killing fields $k^1,k^2,k^3,k^4,k^5,k^6$.
On the other hand, these exponential coordinates also provide a natural section of the projection $K\to K/H$ and therefore, we can speak about left-invariant vector fields in these exponential coordinates and obtain explicit coordinate formulae for them.

\begin{lem}\label{c-frame}
The left-invariant vector fields $e^1,e^2,e^3,e^4$ corresponding to $ k^1,k^2,k^3,k^4$ and the dual coframe $e_1^{*},e_2^{*},e_3^{*},e_4^{*}$ (do not confuse with the dual frame $ e_1,e_2,e_3,e_4 $ w.r.t. the metric) take the following form in the exponential coordinates:

\begin{align*}
e^1&=\partial_{x_1}+\gamma x_3\partial_{x_2}-\gamma x_2\partial_{x_3}-(\frac12a_1x_2^2+\frac12a_2x_3^2-x_4\epsilon)\partial_{x_4},\ \
e^2=\partial_{x_2},\ \ e^3=\partial_{x_3},\ \ e^4= \partial_{x_4}\\
e_1^*&=dx_1, \ \ e_2^*=-\gamma x_3dx_1+dx_2, \ \ 
e_3^*=\gamma x_2dx_1+dx_3,\ \ 
e_4^*=(\frac12a_1x_2^2+\frac12a_2x_3^2-x_4\epsilon)dx_1+dx_4.
\end{align*}
\end{lem}
\begin{proof}
We extend the exponential coordinates to $K$ by considering $\ell=\exp(k^1 x_1)\exp(x_2k^2+x_3k^3)\exp(x_4k^4)\exp(x_5k^5+x_6k^6).$
We observe that the Maurer-Cartan form $\omega_K=\ell^{-1}d\ell$ of $K$ can be expressed as
\begin{align*}
\omega_K&= \Ad(\exp(-x_5k^5-x_6k^6)\exp(-x_4k^4)\exp(-x_2k^2-x_3k^3))(k^1 dx_1)+\omega_{heis}
\end{align*}
  in the exponential coordinates, where $\omega_{heis}$ is the Maurer-Cartan form of the Heisenberg part generated by $k^2,k^3,k^4,k^5,k^6$. This directly leads to formulae for $e^1,e^2,e^3,e^4$ and their duals when we restrict ourselves to the exponential coordinates downstairs, i.e., the subset given by $x_5=x_6=0$.
\end{proof}
The homogeneous plane wave metric $g_{\epsilon,a_1,a_2,\gamma}$ is then given in the exponential coordinates using the dual basis $e_1^{*},e_2^{*},e_3^{*},e_4^{*}$ by
\begin{align*}
2e_1^{*}e_4^{*}+(e_2^{*})^2+(e_3^{*})^2&=((\gamma^2+a_1)x_2^2+(\gamma^2+a_2)x_3^2-2\epsilon x_4)dx_1^2-2\gamma x_3 dx_1dx_2+2\gamma x_2 dx_1dx_3
\\ &+2dx_1dx_4+dx_2^2+dx_3^2.
\end{align*}

The isometry from Brinkmann coordinates to exponential coordinates takes for $\epsilon=0$ form
\begin{align*}
%\textrm{\ for $\epsilon=0$: } 
x_1&=x_+,\:\: x_2=\cos(\gamma x_+)z_1+\sin(\gamma x_+)z_2,\: \: 
x_3=-\sin(\gamma x_+)z_1+\cos(\gamma x_+)z_2,\: \: x_4=x_-
\end{align*}
and for $\epsilon=1$ takes form
\begin{align*}
%\textrm{\ for $\epsilon=1$: } 
x_1&=\ln(x_+),\:\: x_2=\cos(\gamma \ln(x_+))z_1+\sin(\gamma \ln(x_+))z_2,\\
x_3&=-\sin(\gamma \ln(x_+))z_1+\cos(\gamma \ln(x_+))z_2,\:\: x_4=x_+x_-.
\end{align*}

For the realization of the conformal equivalence  $g_{1,a_1,a_2,\gamma} \simeq g_{0,a_1+\frac14,a_2+\frac14,\gamma}$, we consider the following conformal isometry
\begin{align*}
x_+&=\exp(x_+),\:\: x=\exp(\frac{x_+}{2})z_1,\:\: y=\exp(\frac{x_+}{2})z_2,\:\:v=x_--\frac{z_1^2+z_2^2}{4}
\end{align*}
from \cite{HS,AG}. 
For the realization of the isometries  $g_{0,a_1,a_2,\gamma}\simeq g_{0,\lambda^2a_1,\lambda^2a_2,\lambda\gamma}$ and $g_{0,a_1,a_2,\gamma} \simeq g_{0,a_2,a_1,\gamma}$, we consider
\begin{align*}
x_+&=\lambda x_+,\: \: \: x=z_1,\: \:y=z_2,\: \:v=\frac{x_-}{\lambda}, {\rm \ \ \ and \ \ \ }
x_+=-x_+,\: \:x=z_2,\: \:y=z_1,\: \:v=-x_-,
\end{align*}
respectively.

Although there is no simple expression for the Killing fields, we can find the jets of the reducible Killing tensors at the origin using the $\ad$-action, which allows us to characterize the reducible Killing and conformal Killing tensors.

\begin{prop}\label{reducible-prop}
The spaces of reducible conformal Killing tensors have
  \begin{enumerate}
 \item dimension $27$, if $a_1\neq a_2$,

\item dimension $84$, if $a_1= a_2$.
  \end{enumerate}
  
The spaces of reducible Killing tensors (including the metric $g$) have

\begin{enumerate}
\item[(3)] dimension $21$, if $\epsilon=0, a_1\neq a_2$,

\item[(4)] dimension $27$, if $\epsilon=0, a_1= a_2\neq 0$,
  
\item[(5)] dimension $22$, if $\epsilon\neq 0, a_1\neq a_2$, 

\item[(6)]  dimension $28$, if $\epsilon\neq 0, a_1= a_2\neq 0$,
\item[(7)] dimension $50$,  if $a_1= a_2=0$.
  \end{enumerate}
\end{prop}
\begin{proof}
The case $a_1=a_2$ is conformally flat, the case $a_1=a_2=0$ is flat, and cases (2), (7) follow from the general theory of BGG operators. The cases (4), (6) were proven in \cite{KO}. 

If $a_1\neq a_2$, then in addition to the Killing fields $k^1,\dots, k^6$, there is exactly one additional conformal Killing field that we denote $k^7$, which is the homothety of the form
$$-x_2\partial_{x_2}-x_3\partial_{x_3}-2x_4\partial_{x_4}$$
in the exponential coordinates. These provide us all of the reducible Killing and conformal Killing tensors that we parametrize by  $u_{ij}k^ik^j, 1\leq i \leq j\leq 6$ and $u_{ij}(k^ik^j)_0, 1\leq i \leq j\leq 7$, respectively. In the frame $e^1,e^2,e^3,e^4$, their linear combinations take form
$$\sum_{p\leq q\leq 4} 2f_{p,q}e^pe^q:=\sum_{i\leq j\leq 6}2u_{ij}k^ik^j
{\rm \ \ \ and \ \ \ } 
(\sum_{p\leq q\leq 4} 2f_{p,q}e^pe^q)_0:=\sum_{i\leq j\leq 7}2u_{ij}(k^ik^j)_0,$$
for smooth functions $f_{p,q}=f_{p,q}(x_1,x_2,x_3,x_4)$, respectively.

To obtain the claims (1),(3) and (5), we compare the jets of the functions $f_{p,q}$ to determine, which $u_{ij}$ lead to linearly independent Killing and conformal Killing tensors. To obtain the jets of the functions $f_{p,q}$ at the origin $o=(0,0,0,0)$, let us differentiate the action $\exp(-\ad(x_4k^4))\exp(-\ad(x_2k^2+x_3k^3))\exp(-\ad(x_1k^1))$ induced on $k^ik^j$. This leads to
$$f(o)= \left[ \begin {smallmatrix} 2{ u_{11}}&{ u_{12}}&{ u_{13}}&{ u_{14}
}\\ *&2{ u_{22}}&{ u_{23}}&{ u_{24}}
\\ *&*&2{ u_{33}}&{ u_{34}}
\\ *&*&*&2{ u_{44}}
\end {smallmatrix} \right] 
,\ \  \partial_{x_1}f(o)=  \left[ \begin {smallmatrix} 0&{ u_{13}}{ \gamma}-{ u_{15}}&-{ 
u_{12}}{ \gamma}-{ u_{16}}&{ u_{14}}\epsilon\\ *&2{ u_{23}}{ \gamma}-2{ u_{25}}&2\gamma( u_{33}-u_{22})
-{ u_{26}}-{ u_{35}}&{
 u_{24}}\epsilon+{ u_{34}}{ \gamma}-{ u_{45}}\\ 
*&*&-2{ u_{23}}{ \gamma}-2{ u_{36}}&{
 u_{34}}\epsilon-{ u_{24}}{ \gamma}-{ u_{46}}\\ 
*&*&*&4{ u_{44}}
\epsilon\end {smallmatrix} \right],
$$
$$\partial_{x_2}f(o)= \left[ \begin {smallmatrix} 0&-{ u_{17}}&2{ u_{11}}{ \gamma}&{
 u_{15}} \\ *&-2{ u_{27}}&{ u_{12}}{ 
\gamma}-{ u_{37}}&{ u_{25}}-{ u_{47}}\\ *&*&2{ u_{13}}{ \gamma}&
{ u_{14}}{ \gamma}+{ u_{35}}\\ *&*&*&2{ u_{45}}
\end {smallmatrix} \right], \ \ 
\partial_{x_3}f(o)=\left[ \begin {smallmatrix} 0&-2{ u_{11}}{ \gamma}&-{ u_{17}}&{
 u_{16}}\\ *&-2{ u_{12}}{
 \gamma}&-{ u_{13}}{ \gamma}-{ u_{27}}&{ u_{26}}-{ u_{14}}{ \gamma}
\\ *&*&-2{ 
u_{37}}&{ u_{36}}-{ u_{47}}\\ *&*&*&2{ u_{46}}\end {smallmatrix}
 \right], 
$$
$$2\partial_{x_1x_1}f(o)=
 \Bigg[ \begin {smallmatrix} 0&-2{ u_{16}}{ \gamma}+{ u_{12}}
 \left( -{{ \gamma}}^{2}+{ a_1} \right) -{ u_{15}}\epsilon &\cdots\\ 
*&4{ u_{55}}+{ u_{22}} \left( -4{{
 \gamma}}^{2}+4{ a_1} \right) -2{ u_{25}}\epsilon-4{ u_{35}}
{ \gamma}+4{ u_{33}}{{ \gamma}}^{2}-4{ u_{26}}{ \gamma}&\cdots\\ 
*&*&\cdots\\ *&*&\cdots
\end{smallmatrix} 
$$
$$\begin{smallmatrix}
\cdots & { u_{13}} \left( -{{ \gamma}}^{2}+{ a_2} \right) -{ u_{16}}\epsilon+2
{ u_{15}}{ \gamma}&{ u_{14}}{\epsilon}^{2}\\ \cdots &
{
 u_{23}} \left( -4{{ \gamma}}^{2}+{ a_1}+{ a_2} \right) +4{
 u_{25}}{ \gamma}-{ u_{35}}\epsilon+2{ u_{56}}-4{ u_{36}}{
 \gamma}-{ u_{26}}\epsilon&{ u_{24}} \left( {\epsilon}^{2}-{{ 
\gamma}}^{2}+{ a_1} \right) -2{ u_{46}}{ \gamma}+2{ u_{34}}{
 \gamma}\epsilon-3{ u_{45}}\epsilon\\ \cdots &
4{ u_{22}}{{ 
\gamma}}^{2}+4{ u_{35}}{ \gamma}+4{ u_{66}}+{ u_{33}} \left( -4
{{ \gamma}}^{2}+4{ a_2} \right) -2{ u_{36}}\epsilon+4{ 
u_{26}}{ \gamma}&-2{ u_{24}}{ \gamma}\epsilon-3{ u_{46}}
\epsilon+{ u_{34}} \left( {\epsilon}^{2}-{{ \gamma}}^{2}+{ a_2}
 \right) +2{ u_{45}}{ \gamma}\\ \cdots & *&8{ u_{44}}{\epsilon}^{2}
\end{smallmatrix} \Bigg],
$$
$$2\partial_{x_2x_2}f(o)=
\left[ \begin {smallmatrix} 0&0&0&2{ u_{11}}{ a_1}
\\ *&4{ u_{77}}&-2{ u_{17}}{ \gamma}&{ u_{12}}
{ a_1}-2{ u_{57}}\\ *&*&4{ u_{11}}{{ \gamma}}^{2}&{ u_{13}}{ a_1}+2{ u_{15}}{ 
\gamma}\\ *&*&*&2{ u_{14}}
{ a_1}+4{ u_{55}}\end {smallmatrix} \right], \ \  
2\partial_{x_3x_3}f(o)= \left[ \begin {smallmatrix} 0&0&0&2{ u_{11}}{ a_2}
\\ *&4{ u_{11}}{{ \gamma}}^{2}&2{ u_{17}}{
 \gamma}&{ u_{12}}{ a_2}-2{ u_{16}}{ \gamma}
\\ *&*&4{ u_{77}}&{ u_{13}}
{ a_2}-2{ u_{67}}\\ *&*&*&2{ u_{14}}{ a_2}+4{ u_{66}}\end {smallmatrix} \right].$$

These jets provide enough conditions to determine the linear independence of $u_{ij}$, which gives the dimensions $21$ for Killing tensors and $28$ for conformal Killing tensors (including the trace). To obtain the dimensions in the statement, we need to compare these jets with the jets of the metric. From this comparison, we conclude that in the case $\epsilon=0$, $$g=2k^1 k^4+k^2 k^2+k^3 k^3+2\gamma k^2k^6-2\gamma k^3k^5-a_1 k^5k^5-a_2 k^6 k^6,$$
and in the case $\epsilon=1$, $$2k^1 k^4+k^2 k^2+k^3 k^3-k^2k^5+2\gamma k^2k^6-2\gamma k^3k^5-k^3k^6-k^4k^7-a_1 k^5k^5-a_2 k^6 k^6$$
is a functional multiple of $g.$ Therefore, in the case (5), the metric provides additional reducible Killing tensor, while in the case (1), there is one less reducible conformal Killing tensor.
\end{proof}

\section{BGG operators on homogeneous plane waves}\label{sec-bgg-how}
Let us review the method for finding solutions of the first BGG operators on homogeneous conformal and projective geometries, \cite{bgg-my} and \cite{bgg-conf}. The solutions of the first BGG operators on the homogeneous space $K/H$ can be described as sections of the homogeneous vector bundle $K \times_H \mathbb{X}$, where $\mathbb{X}=S^2\fk/\fh$ in the case of Killing tensors, and $\mathbb{X}=S^2_0\fk/\fh$ in the case of conformal Killing tensors. 
With a chosen complement $\fc=\langle k^1,k^2,k^3,k^4 \rangle$ of $\fh$ in $\fk$, the solutions in the exponential coordinates can be described by functions $\fc \to \mathbb{X}$, where $\mathbb{X}=S^2\fc$ or $\mathbb{X}=S^2_0\fc$. These functions are the coordinate functions in the frame of $K\times_H \mathbb{X}$ induced by the frame $e^1,e^2,e^3,e^4$ in Lemma \ref{c-frame}. The main results of \cite[Proposition 5.1]{bgg-my} and \cite[Theorem 1.1]{bgg-conf} state that the solutions are completely described by a pair $(\Phi|_{\fc},\pi)$, where
\begin{enumerate}
    \item $\Phi|_{\fc}$ is restriction to $\fc$ of a representation $\Phi:\fk \to \mathfrak{gl}(\mathbb{S})$ of the Lie algebra $\fk$ on the vector space $\mathbb{S}$ (representing the space of solutions), and 
    \item $\pi:\mathbb{S}\to \mathbb{X}$ is a linear map.
\end{enumerate}
The functions $\fc \to \mathbb{X}$ corresponding to $v\in \mathbb{S}$ are then given by
\[(x_1,x_2,x_3,x_4)\mapsto \pi(\exp(-\Phi(x_4k^4))\exp(-\Phi(x_2k^2+x_3k^3))\exp(-\Phi(x_1k^1))v).\]

In the following subsections, we summarize how to compute explicitly the pair $(\Phi|_{\fc},\pi)$:

\subsection*{(1)}
The first step is to find the linear maps $\alpha:\fk\to \mathfrak{so}(2,4)$ and $\beta:\fk\to \mathfrak{sl}(5,\mathbb{R})$ that form the basic components of the representations $\Phi$. These depend only on the geometries in question and are the same for all BGG operators. For homogeneous plane waves, we find $\alpha,\beta$ in the following two Propositions \ref{cnormal} and \ref{pnormal}.

\begin{prop}\label{cnormal}
The map $\alpha:\fk\to \mathfrak{so}(2,4)$  for the homogeneous plane wave $g_{\epsilon,a_1,a_2,\gamma}$ is
\begin{align*}
\alpha(x_ik^i)=\left[ 
\begin {smallmatrix} 
0& \frac12 ( a_1+a_2) x_1&0&0&0&0\\ 
x_1&\epsilon x_1&0&0&0&0\\ 
x_2&-x_5&0&-\gamma x_1&0&0\\ 
x_3&-x_6&\gamma x_1&0&0&0\\ 
x_4&0& x_5&x_6&-\epsilon x_1&- \frac12( a_1+a_2) x_1\\
 0&-x_4&-x_2&-x_3&-x_1&0
\end {smallmatrix} \right],
\end{align*}
where $\mathfrak{so}(2,4)$ preserves the scalar product $\langle u,v \rangle=u_1v_6+v_6u_1+u_2v_5+v_5u_2+u_3v_3+u_4v_4.$
\end{prop}
\begin{proof}
Following \cite[Proposition 2.1.]{bgg-conf}, the first column and the last row are determined by our choice of the basis of $\fc$. Then we use the Lie brackets of $\fk$ to compute the image of $x_5,x_6$. The remaining parts of $\alpha$ are computed by considering a general linear map $\fc$ to the rest of $\mathfrak{so}(2,4)$ and imposing the normalization conditions from \cite[Proposition 2.1.]{bgg-conf} on the tensor $(X,Y)\mapsto [\alpha(X),\alpha(Y)]-\alpha([X,Y]), X,Y\in\fk$. This uniquely determines $\alpha$.
\end{proof}

\begin{prop}\label{pnormal}
The map $\beta:\fk\to \mathfrak{sl}(5,\mathbb{R})$ for the homogeneous plane wave $g_{\epsilon,a_1,a_2,\gamma}$ is
\begin{align*}
\beta_{\epsilon,a_1,a_2,\gamma}(x_ik^i)=\left[ 
\begin {smallmatrix} 
0& \frac13 ( a_1+a_2) x_1&0&0&0\\ 
x_1&\epsilon x_1&0&0&0\\ 
x_2&-x_5&0&-\gamma x_1&0\\ 
x_3&-x_6&\gamma x_1&0&0\\ 
x_4&0& x_5&x_6&-\epsilon x_1
\end {smallmatrix} \right] 
\end{align*}
\end{prop}
\begin{proof}
Following  \cite[Section 2.1]{bgg-my}, it remains to determine the projective Schouten tensor, i.e., the first row of the image of $\beta$. The remaining part of $\beta$ is calculated by considering a general linear map $\fc$ to the rest of $\mathfrak{sl}(5,\mathbb{R})$ and imposing the normalization conditions \cite[Section 2.1]{bgg-my} on the tensor $(X,Y)\mapsto [\beta(X),\beta(Y)]-\beta([X,Y]),X,Y\in\fk$. This uniquely determines $\beta$.
\end{proof}

\subsection*{(2)}\label{ss-sec}

Another component of $\Phi$ is a representation $\rho: \mathfrak{so}(2,4)\to \mathfrak{gl}(\mathbb{V})$ or $\rho: \mathfrak{sl}(5,\mathbb{R})\to \mathfrak{gl}(\mathbb{V})$, respectively, associated with the particular BGG operator.
For Killing and conformal Killing tensors, the representations $\rho$ are known, and we describe them in Sections \ref{sec-CKT}, \ref{sec-fKT}, and \ref{sec-KT}, in detail. The second step is then to compute the remaining part $\Psi: \fc\to \mathfrak{gl}(\mathbb{V})$ so that
$$\Phi:=\rho\circ \alpha+\Psi \rm{\ \ \ or \ \ \ } \Phi:=\rho\circ \beta+\Psi,$$
respectively.
For the calculation of $\Psi$, we follow the algorithm in \cite[Section 3.3]{bgg-my} and \cite[Section 3.3]{bgg-conf}. However, these algorithms are too technical to present them here in detail. Moreover, to verify that we obtained the correct result, it is sufficient, according to \cite{new-norm}, to check that $\partial^* R^{\Phi}(s)=0$ for all $s\in \mathbb{V}$, where $ R^{\Phi}(X,Y)=[\Phi(X),\Phi(Y)]-\Phi([X,Y])$,   $X,Y\in \fk$ is the curvature form of $\Phi$ and $\partial^*$ is the Kostant's codifferential.

\subsection*{(3)}\label{ts-cec}

The third step is to compute the maximal $\Phi$-invariant subspace $\mathbb{S}$ annihilated by the curvature $R^{\Phi}$ of $\Phi$. This is an iterative process that starts with the subspace $\mathbb{S}^0\subset \mathbb{V}$ annihilated by the curvature $R^{\Phi}$. Then we continue iteratively to determine the subspaces $\Phi(\fk)\mathbb{S}^{i+1}\subset \mathbb{S}^i$ until $\mathbb{S}^{i+1}=\mathbb{S}^i=\mathbb{S}$ holds after a finite number of steps. If we disregard the parameters, then the iterative process involves only solving linear equations. This can be done with an arbitrary computer algebra system; we use the computer algebra system Maple, and these computations can be found in the Maple file \cite{bgg-comp} accompanying the arXiv submission. This provides us with the spaces $\mathbb{S}^{gen}$ of generic solutions given in Propositions \ref{generic-CKT}, \ref{generic-fKT}, and \ref{generic-KT}. Then, depending on the values of parameters, the linear equations change, and we investigate for which values of parameters there are additional solutions. The generic solutions allow us to search for these solutions in complements $C$ to the subspaces of the generic solutions. We present the results in Sections \ref{sec-CKT}, \ref{sec-fKT}, and \ref{sec-KT}.

\section{Conformal Killing tensors - Proof of Theorem \ref{main-CKT}}\label{sec-CKT}

We need to consider here two special cases. If $a_1=a_2$, then we have the conformally flat case, and there are $84$ reducible conformal Killing tensors. If $a_1\neq a_2$, then we can assume $\epsilon=0$ without loss of generality. The following proposition provides the formulae for the irreducible conformal Killing tensors from Theorem \ref{main-CKT}. It is a direct computation to check that these tensors are conformal Killing tensors. To show that they are irreducible and that we have all of them, we show here that they come from a complement of the space of the reducible conformal Killing tensors we described in Proposition \ref{reducible-prop}. 

\begin{prop}\label{expicit-CKT}
The homogeneous plane waves $g_{0,\pm\frac{8}{3},\pm\frac{2}{3},0}$ have nine-dimensional spaces of irreducible conformal Killing tensors as follows, where 
\begin{itemize}
\item ${\bf s}(q)=\sin(q\sqrt{6}x_+)$, ${\bf c}(q)=\cos(q\sqrt{6}x_+)$  in the case of $-$, and 
\item ${\bf s}(q)=\sinh(q\sqrt{6}x_+)$, ${\bf c}(q)=\cosh(q\sqrt{6}x_+)$ in the case of of $+$:
\end{itemize}

\begin{gather*}
2\big(c_5 {\bf s}( {\scriptstyle \frac{4}{3}} )+ c_6 {\bf c}({\scriptstyle \frac{4}{3}})\big)
\boldsymbol{\partial}_{x_+}^2
+{\scriptstyle \frac{8\sqrt{6}}{3}}  z_1 \big(c_5{\bf c}({\scriptstyle \frac{4}{3}}) \pm c_6{\bf s}({\scriptstyle \frac{4}{3}} ) \big)
\boldsymbol{\partial}_{x_+}\boldsymbol{\partial}_{z_1}
+\big({\scriptstyle \frac{4\sqrt{6}}{3}}z_2( c_5 {\bf c}({\scriptstyle \frac{4}{3}})   \pm c_6{\bf s}({\scriptstyle \frac{4}{3}} )  )
\\
+2 z_1(      c_8{\bf c}({\scriptstyle \frac{1}{3}})+c_7{\bf s}({\scriptstyle \frac{1}{3}}))
+2z_2 (c_9 {\bf c}({\scriptstyle \frac{2}{3}})  + c_1{\bf s}({\scriptstyle \frac{2}{3}})  )
+2 (c_2 {\bf s}(1)+ c_3 {\bf c}(1))\big)
\boldsymbol{\partial}_{x_+}\boldsymbol{\partial}_{z_2}
\\
+\big( 
 \mp {\scriptstyle \frac{32}{3}} z_1^2  (c_6 {\bf c}({\scriptstyle \frac{4}{3}}) + c_5{\bf s}({\scriptstyle \frac{4}{3}}))+{\scriptstyle \frac{2\sqrt{6}}{3}}z_2(    c_2{\bf c}(1)  
\pm c_3 {\bf s}(1)   )
\big)
\boldsymbol{\partial}_{x_+}\boldsymbol{\partial}_{x_-}
+\big(-
 {\scriptstyle \frac{4\sqrt{6}}{3}}x_- (c_5 {\bf c}({\scriptstyle \frac{4}{3}})\pm c_6 {\bf s}({\scriptstyle \frac{4}{3}})   )
 \\
 \pm{\scriptstyle \frac{8}{3}}  z_1^2(c_6 {\bf c}({\scriptstyle \frac{4}{3}})
 + c_5{\bf s}({\scriptstyle \frac{4}{3}}) )
- {\scriptstyle \frac{ \sqrt{6}}{3}}z_1 z_2( c_7{\bf c}({\scriptstyle \frac{1}{3}})  \pm c_8 {\bf s}({\scriptstyle \frac{1}{3}})  )
\big)
\boldsymbol{\partial}_{z_1}^2
+\big( {\scriptstyle \frac{2\sqrt{6}}{3}} ( z_1^2- z_2^2)(c_7{\bf c}({\scriptstyle \frac{1}{3}}) 
\pm  c_8{\bf s}({\scriptstyle \frac{1}{3}}))
\\
-2 x_- (c_8 {\bf c}({\scriptstyle \frac{1}{3}})
 + c_7  {\bf s}({\scriptstyle \frac{1}{3}}))
+
{\scriptstyle \frac{4 \sqrt{6}}{3}}z_1 z_2(\pm c_9{\bf s}({\scriptstyle \frac{2}{3}})  + c_1 {\bf c}({\scriptstyle \frac{2}{3}})  )
\pm {\scriptstyle \frac{16}{3}}  z_1 z_2 (c_6 {\bf c}({\scriptstyle \frac{4}{3}}) 
+  c_5 {\bf s}({\scriptstyle \frac{4}{3}}))
\\
+{\scriptstyle \frac{4\sqrt{6}}{3}}z_1( c_2 {\bf c}(1)   \pm c_3 {\bf s}(1)   )
+2 c_4 z_2\big)
\boldsymbol{\partial}_{z_1}\boldsymbol{\partial}_{z_2}
+\big(-{\scriptstyle \frac{2\sqrt{6}}{3}} z_2 x_- (\pm c_8 {\bf s}({\scriptstyle \frac{1}{3}}) + c_7{\bf c}({\scriptstyle \frac{1}{3}})) 
\\
\pm{\scriptstyle \frac{4}{3}} z_2  z_1^2(c_8 {\bf c}({\scriptstyle \frac{1}{3}}) 
+ c_7{\bf s}({\scriptstyle \frac{1}{3}})) 
-{\scriptstyle \frac{32\sqrt{6}}{9}}  z_1^3 (\pm c_5 {\bf c}({\scriptstyle \frac{4}{3}})
+    c_6 {\bf s}({\scriptstyle \frac{4}{3}}))
\pm {\scriptstyle \frac{32}{3}} z_1x_-(c_5 {\bf s}({\scriptstyle \frac{4}{3}}) +  c_6  {\bf c}({\scriptstyle \frac{4}{3}}))
\\
\pm {\scriptstyle \frac{8}{3}} z_1 z_2 ( c_3 {\bf c}(1)+ c_2 {\bf s}(1))\big)
\boldsymbol{\partial}_{z_1}\boldsymbol{\partial}_{x_-}
+\big(-  {\scriptstyle \frac{2\sqrt{6}}{3}} (z_1^2-z_2^2)(\pm c_9{\bf s}({\scriptstyle \frac{2}{3}}) + c_1 {\bf c}( {\scriptstyle \frac{2}{3}})) - 2  x_-( c_9{\bf c}( {\scriptstyle \frac{2}{3}})  +  c_1{\bf s}({\scriptstyle \frac{2}{3}}) )
\\
\pm {\scriptstyle \frac{8}{3}}  z_2^2(c_6{\bf c}({\scriptstyle \frac{4}{3}}) 
 + c_5{\bf s}({\scriptstyle \frac{4}{3}})) 
+ z_1 z_2 ( c_7 {\bf c}({\scriptstyle \frac{1}{3}})  \pm c_8 {\bf s}({\scriptstyle \frac{1}{3}})   )
+{\scriptstyle \frac{4  \sqrt{6}}{3}} z_2( c_2 {\bf c}(1)  \pm  c_3 {\bf s}(1) )
-2 c_4 z_1\big)
\boldsymbol{\partial}_{z_2}^2
\\
+\big( {\scriptstyle \frac{4\sqrt{6}}{3}} z_2x_-  (c_1{\bf c}({\scriptstyle \frac{2}{3}})  \pm  c_9 {\bf s}({\scriptstyle \frac{2}{3}}))
\mp {\scriptstyle \frac{4}{3}} z_2(2  z_1^2+ z_2^2)(c_9{\bf c}({\scriptstyle \frac{2}{3}})+ c_1{\bf s}({\scriptstyle \frac{2}{3}}) )
+   {\scriptstyle \frac{4 \sqrt{6}}{3}} z_1 x_- (c_7  {\bf c}({\scriptstyle \frac{1}{3}}) \pm c_8  {\bf s}({\scriptstyle \frac{1}{3}})) 
\\
-{\scriptstyle \frac{32 \sqrt{6} }{9}}  z_2  (z_1^2+{\scriptstyle \frac{1}{4}} z_2^2)  (\pm c_5{\bf c}({\scriptstyle \frac{4}{3}})  + c_6 {\bf s}({\scriptstyle \frac{4}{3}}) )
\mp {\scriptstyle \frac{16}{3}} (z_1^2+{\scriptstyle \frac{1}{2}} z_2^2) (c_3 {\bf c}(1)+ c_2 {\bf s}(1))
\\
\mp {\scriptstyle \frac{8}{3}}  z_1^3 (c_8{\bf c}({\scriptstyle \frac{1}{3}})
   + c_7  {\bf s}({\scriptstyle \frac{1}{3}}))  \big)
\boldsymbol{\partial}_{z_2}\boldsymbol{\partial}_{x_-}
+\big(-  {\scriptstyle \frac{\sqrt{4}}{9}}  z_1^2 z_2^2  (\pm c_1 {\bf c}({\scriptstyle \frac{2}{3}})   + c_9 {\bf s}({\scriptstyle \frac{2}{3}} )) \mp
{\scriptstyle \frac{4}{3}}  x_-  z_2^2 (c_9{\bf c}({\scriptstyle \frac{2}{3}})  
+  c_1 {\bf s}({\scriptstyle \frac{2}{3}}) )
\\
-  {\scriptstyle \frac{ 32\sqrt{6}}{9}} x_-  z_1^2 (\pm c_5
 {\bf c}({\scriptstyle \frac{4}{3}})   +
 c_6 
 {\bf s}({\scriptstyle \frac{4}{3}}))
 +   {\scriptstyle \frac{64}{9}} z_1^4 (c_6 {\bf c}({\scriptstyle \frac{4}{3}}) 
+ c_5  {\bf s}({\scriptstyle \frac{4}{3}}) )
+{\scriptstyle \frac{2\sqrt{6}}{9}} z_1 z_2^3 (\pm c_7 {\bf c}({\scriptstyle \frac{1}{3}})
+ c_8 {\bf s}({\scriptstyle \frac{1}{3}})) 
\\
\pm {\scriptstyle \frac{8}{3}} z_1  z_2x_-(c_7 {\bf s}({\scriptstyle \frac{1}{3}})  +c_8 {\bf c}({\scriptstyle \frac{1}{3}}))
-{\scriptstyle \frac{16 \sqrt{6} }{9}} z_1^2z_2 ( \pm c_2  {\bf c}(1)
+c_3  {\bf s}(1) )\mp{\scriptstyle \frac{4}{3}} c_4 z_1 z_2^2\big)
\boldsymbol{\partial}_{x_-}^2
\\
+({\scriptstyle \frac{\sqrt{6}}{6}}(z_1^2-z_2^2)(c_1{\bf c}({\scriptstyle \frac{2}{3}}) \pm c_9{\bf s}({\scriptstyle \frac{2}{3}})) 
+{\scriptstyle \frac{\sqrt{6}}{3}}x_-(c_5{\bf c}({\scriptstyle \frac{4}{3}})\pm c_6{\bf s}({\scriptstyle \frac{4}{3}}))
\pm ({\scriptstyle \frac{2}{3}}z_1^2-z_2^2)(c_6{\bf c}({\scriptstyle \frac{4}{3}})+ c_5{\bf s}({\scriptstyle \frac{4}{3}})) \\
+{\scriptstyle \frac12 } x_-(c_9 {\bf c}({\scriptstyle \frac{2}{3}})- c_1 {\bf s}({\scriptstyle \frac{2}{3}}))
-{\scriptstyle \frac{\sqrt{6}}{6}} z_1  z_2(c_7 {\bf c}({\scriptstyle \frac{1}{3}})
\pm c_8  {\bf s}({\scriptstyle \frac{1}{3}}))
-{\scriptstyle \frac{\sqrt{6}}{2}}z_2( c_2 {\bf c}(1)  \pm c_3 {\bf s}(1))+{\scriptstyle \frac12} c_4  z_1)g_{0,\pm\frac{8}{3},\pm\frac{2}{3},0}^{-1}
\end{gather*}

\newpage
The homogeneous plane wave $g_{0,-1,3 ,1}$ has a two-dimensional space of irreducible conformal Killing tensors as follows, where ${\bf c}(q)=\cos(qx_+)$ and ${\bf s}(q)=\sin(qx_+)$:

\begin{gather*}
\big(- c_1(z_2{\bf c}(2)-z_1{\bf s}(2)-2(x_+z_1-z_2))
-c_2x_+(-x_+^3{\bf c}(1)+(2x_+^2+3){\bf s}(1))
 \big)  
 \boldsymbol{\partial}_{x_+}\boldsymbol{\partial}_{z_1}
 \\
 +\big(- c_1(z_1{\bf c}(2)+z_2{\bf s}(2)-2(x_+z_2+z_1))
+c_2x_+((2x_+^2+3){\bf c}(1)+x_+^3{\bf s}(1))
 \big)  
 \boldsymbol{\partial}_{x_+}\boldsymbol{\partial}_{z_2}
 \\
 +\big(c_1(-2(z_1^2-z_2^2){\bf c}(2)-4z_1z_2{\bf s}(2)-z_1^2-z_2^2+4x_-x_+)
 -c_2((2x_+z_1(2x_+^2-3)
 \\
 +3z_2(2x_+^2-1)){\bf c}(1)
 +(3z_1(1-2x_+^2)-2x_+z_2(3-2x_+^2))){\bf s}(1))
  \big)\boldsymbol{\partial}_{x_+}\boldsymbol{\partial}_{x_-}
 +\big(-c_1(x_-{\bf s}(2) 
 \\
 +2x_+z_1z_2)
 -{\scriptstyle \frac{1}{2}}c_2x_+((x_+^3z_2+z_1(2x_+^2+9)){\bf c}(1)+(x_+^3z_1-z_2(2x_+^2+9)){\bf s}(1))
 \big)  \boldsymbol{\partial}_{z_1}^2
 +\big( 2c_1(x_-{\bf c}(2)
 \\
 +x_+(z_1^2-z_2^2))
 +c_2x_+((x_+^3z_1-z_2(2x_+^2+9)){\bf c}(1)-(x_+^3z_2+z_1(2x_+^2+9)){\bf s}(1))
 \big)  \boldsymbol{\partial}_{z_1}\boldsymbol{\partial}_{z_2}
 \\
 +\big( {\scriptstyle \frac{1}{3}}c_1( z_2(z_2^2-3z_1^2){\bf c}(4)+z_1(z_1^2-3z_2^2){\bf s}(4)+(2z_2(2z_2^2-3z_1^2)+6x_-z_1){\bf c}(2)+(z_1(z_1^2-9z_2^2)
 \\
 +6x_-z_2){\bf s}(2)
 -6z_2(2x_-x_+-z_1^2-z_2^2))
 +{\scriptstyle \frac{1}{2}}c_2(( 2z_2^2 x_+^2(x_+^2+6)+4x_+^3z_1z_2
 +6x_+(z_1z_2-4x_-)
 \\
 +9(z_1^2+z_2^2)){\bf c}(1)
 +(-2x_+^4z_1z_2+4z_2x_+^2(z_2x_+-3z_1  )-6x_+(z_2^2+2z_1^2)+12x_-){\bf s}(1))
 \big)  \boldsymbol{\partial}_{z_1}\boldsymbol{\partial}_{x_-}
 \\
  +\big(c_1(x_-{\bf s}(2)+2x_+z_1z_2)+{\scriptstyle \frac{1}{2}}c_2x_+(( x_+^2(x_+z_2+2z_1)+9z_1){\bf c}(1)+( x_+^2(x_+z_1-2z_2)-9z_2){\bf s}(1))
 \big)  \boldsymbol{\partial}_{z_2}^2
 \\
  +\big( {\scriptstyle \frac{1}{3}}c_1( z_1(3z_2^2-z_1^2){\bf c}(4)+z_2(z_2^2-3z_1^2){\bf s}(4)+(2z_1(2z_1^2-3z_2^2)-6x_-z_2){\bf c}(2)+(z_2(9z_1^2-z_2^2)
  \\
  +6x_-z_1){\bf s}(2)-6z_1^3+6(2x_-x_+-z_2^2)z_1)
 +{\scriptstyle \frac{1}{2}}c_2(-(2z_1z_2 x_+^2(x_+^2 +6)+4z_1^2x_+^3 -6x_+(z_1^2+2z_2^2)
 \\+12x_-){\bf c}(1)+(2z_1^2 x_+^2(x_+^2+6)-4x_+^3z_1z_2-6x_+(z_1z_2+4x_-)
 +9(z_1^2+z_2^2)){\bf s}(1)
 )
 \big)  \boldsymbol{\partial}_{z_2}\boldsymbol{\partial}_{x_-}
 \\
+\big({\scriptstyle \frac{1}{12}}c_1
 (-7(z_1^4-6z_1^2z_2^2+z_2^4){\bf c}(4)
 -28z_1z_2(z_1^2-z_2^2){\bf s}(4)
 +(4(z_1^4-z_2^4)-72x_-z_1z_2){\bf c}(2)
 \\
  +(8z_1z_2(z_1^2+z_2^2)+36x_-(z_1^2-z_2^2)){\bf s}(2)
 -9(z_1^2+z_2^2)^2 -24x_-^2)
 +{\scriptstyle \frac{1}{2}}c_2(
 (x_+z_1(3-2x_+^2)(z_1^2-3z_2^2)
 \\
 +z_2(x_+^4+4)(3z_1^2-z_2^2)){\bf c}(3)
 +(x_+z_2(3-2x_+^2)(3z_1^2-z_2^2)-z_1(x_+^4+4)(z_1^2-3z_2^2)){\bf s}(3)
 \\
 +(-z_2(2x_+^4+15)(z_1^2+z_2^2)+12x_-(z_1+2x_+z_2)){\bf c}(1)
 +(z_1(15+2x_+^4)(z_1^2+z_2^2)
 \\
 +12x_-(z_2-2x_+z_1)){\bf s}(1)
 \big)  \boldsymbol{\partial}_{x_-}^2
 + {\scriptstyle \frac{1}{4}}\big(
c_1(2(z_1^2-z_2^2){\bf c}(2)+4z_1z_2{\bf s}(2)-4x_-x_++z_1^2+z_2^2)
 \\
 +c_2((2x_+^2(2x_+z_1+3z_2)-3(2x_+z_1+z_2)){\bf c}(1)+(2x_+^2(2x_+z_2-3 z_1)-3(2x_+z_2-z_1)){\bf s}(1))
 \big)g_{0,-1,3 ,1}^{-1}
\end{gather*}
\end{prop}

We know from \cite{bgg-conf} that the representation $\rho: \mathfrak{so}(2,4)\to \mathfrak{gl}(\mathbb{V})$ is the Cartan component $\mathbb{V}\subset S^2 \mathfrak{so}(2,4).$ The space $S^2 \mathfrak{so}(2,4)$ is of dimension $120$, corresponds to positions $(i,j,k,l)$ for $i<j$, $k<l$, $(i,j)\leq (k,l)$, and is embedded in $\otimes^4 \mathbb{R}^6$ via the symmetries of the tensors in $S^2 \mathfrak{so}(2,4)$. Recall that $S^2 \mathfrak{so}(2,4)$ decomposes into distinguished eigenspaces for the $\rho$-action of the grading element in $\mathfrak{so}(2,4)$, \cite{parabook}, with eigenvalues shifted to $0,\dots,4$ according to our conventions. We order the positions $(i,j,k,l)$ in the following way, which describes this grading explicitly:
\begin{enumerate}
\item Firstly, we order 
\begin{enumerate}
\item[(i)] $(1,j,1,l)$, $1<j\leq l<6$ with variables $w_1,\dots,w_{10}$  
\end{enumerate}
corresponding to the highest grading component $S^2 \mathfrak{so}(2,4)_4$.
\item Next, we order 
\begin{enumerate}
\item[(ii)] $(1,j,k,l)$, $1<j,k,l<6$ with variables $w_{11},\dots,w_{34}$, and 
\item[(iii)] $(1,j,1,6)$, $1<j<6$ with variables $w_{35},\dots,w_{38}$ 
\end{enumerate}
corresponding to the grading component $S^2 \mathfrak{so}(2,4)_3$.
\item Next, we order 
\begin{enumerate}
\item[(iv)] $(1,j,k,6)$, $1<j,k<6$, with variables $w_{39},\dots,w_{54}$, 
\item[(v)] $(1,6,k,l)$, $1<k,l<6$,  with variables $w_{55},\dots,w_{60}$,
\item[(vi)] $(1,6,1,6)$,  with variable $w_{61}$,  and 
\item[(vii)] $(i,j,k,l)$, $1<i,j,k,l<6$  with variables $w_{62},\dots, w_{82}$
\end{enumerate}
corresponding to the grading component $S^2 \mathfrak{so}(2,4)_2$.
\item Next, we order 
\begin{enumerate}
\item[(viii)] $(i,j,k,6)$, $1<i,j,k<6$, with variables $w_{83},\dots,w_{106}$, and 
\item[(ix)]$ (1,6,k,6)$, $1<k<6$ with variables $w_{107},\dots,w_{110}$
\end{enumerate}
corresponding to the grading component $S^2 \mathfrak{so}(2,4)_1$.
\item Finally, we order 
\begin{enumerate}
\item[(x)] $(i,6,k,6),1<i\leq k<6$ with variables $w_{111},\dots,w_{120}$
\end{enumerate}
corresponding to the grading component $\mathbb{X}=S^2 \mathfrak{so}(2,4)_0$.
The projection $\pi: S^2 \mathfrak{so}(2,4)\to S^2\fc$ is $\pi(v)=v_{i,6,k,6}e^{i-1}\otimes e^{j-1}.$
\end{enumerate}

The Cartan component $\mathbb{V}$ is $84$-dimensional, and we describe it (using the Casimir element) as the subspace of $S^2 \mathfrak{so}(2,4)$ with the following $84$ free variables 
$$w_1 \dots, w_{3}, w_{5},\dots, w_{18}, w_{27},\dots, w_{38},w_{55},\dots, w_{65},  w_{68},\dots,w_{90}, w_{99},\dots, w_{113}, w_{115}, \dots w_{120},$$
and  the dependent:
\begin{multicols}{2}
\begin{align*}
w_4&=-{\scriptstyle \frac12}(w_{5}+w_{8})
\\
w_{19}&=w_{15}+w_{29}
\\
w_{20}&=w_{16}-w_{30}-w_{37}
\\
w_{21}&=-w_{28}-w_{31}-w_{38}
\\
w_{22}&=w_{27}-w_{32}
\\
w_{23}&=-w_{14}+w_{18}
\\  
w_{24}&=-w_{13}-w_{17}+w_{35}
\\
w_{25}&=w_{16}+w_{30}
\\
w_{26}&=-w_{15}+w_{29}+w_{36}
\\
w_{39}&={\scriptstyle \frac12}(w_{62}+w_{68})  
\\
w_{40}&={\scriptstyle \frac12}(w_{55}-w_{64}+w_{70})
\\
w_{41}&=-{\scriptstyle \frac12}(w_{56}+w_{65}+w_{69})
\\
w_{42}&=-{\scriptstyle \frac14}(w_{57}+w_{61}+w_{73}+w_{77})
\\
w_{43}&={\scriptstyle \frac12}(w_{55}-w_{64}+w_{70}) 
\\
w_{44}&=-{\scriptstyle \frac12}w_{61} +w_{72}+{\scriptstyle \frac12}w_{73}
\\
w_{45}&=-{\scriptstyle \frac12}w_{58}-w_{71}+{\scriptstyle \frac12}w_{74} 
\\
w_{46}&=-{\scriptstyle \frac12}(w_{59}+w_{75}+w_{79})
\\
w_{47}&={\scriptstyle \frac12}(w_{56}-w_{65}-w_{69})
\end{align*}
\begin{align*}
\\
w_{48}&={\scriptstyle \frac12}w_{58}+{\scriptstyle \frac12}w_{74}-w_{71}
\\
w_{49}&=-w_{72}+{\scriptstyle \frac12}w_{77}
\\
w_{50}&={\scriptstyle \frac12}(-w_{60}-w_{76}+w_{78})
\\ 
w_{51}&={\scriptstyle \frac12}w_{57}-{\scriptstyle \frac14}w_{61}-{\scriptstyle \frac14}w_{73}-{\scriptstyle \frac14}w_{77}
\\ 
w_{52}&={\scriptstyle \frac12}(w_{59}-w_{75}-w_{79})
\\
w_{53}&={\scriptstyle \frac12}(w_{60}-w_{76}+w_{78})
\\
w_{54}&={\scriptstyle \frac12}(w_{80}+w_{82})
\\  
w_{66}&={\scriptstyle \frac12}w_{61}-w_{72}-{\scriptstyle \frac12}w_{73}+{\scriptstyle \frac12}w_{77}
\\
w_{67}&=w_{71}-w_{74}
\\
w_{91}&=-w_{84}-w_{89}-w_{107}
\\
w_{92}&=w_{86}+w_{99}
\\
w_{93}&=w_{90}+w_{103}
\\
w_{94}&=-w_{100}-w_{105}+w_{110}
\\
w_{95}&=-w_{85}+w_{88}
\\
w_{96}&=-w_{90}+w_{103}+w_{109}
\\
w_{97}&=w_{86}-w_{99}-w_{108}
\\
w_{98}&=w_{101}-w_{104}
\\  
w_{114}&=-{\scriptstyle \frac12}(w_{115}+w_{118})
\end{align*}
\end{multicols}

As described in Section \ref{ss-sec}, we find the prolongation connection $\Phi.$ The representation $\rho$ is induced by the standard action on $\otimes^4 \mathbb{R}^6$ and $\alpha$ is given in Proposition \ref{cnormal}. The components $\Psi$ of the prolongation connection are as follows, where $A:=a_1-a_2$ (a direct check that these formulae are correct can also be found in the Maple file \cite{bgg-comp} accompanying the arXiv submission):

\begin{multicols}{2}
  \begin{align*}  
       \Psi(e^1)_2 &= {\scriptstyle \frac{A}{30}} w_{11} \\
    \Psi(e^1)_3 &= -{\scriptstyle \frac{A}{30}}  w_{12} \\
    \Psi(e^1)_5 &= -{\scriptstyle \frac{A}{60}} ( 21 w_{13} - 2 w_{17} + 7 w_{35}+2 \gamma w_{63})  
    \\
    \Psi(e^1)_6 &= {\scriptstyle \frac{A}{10}} ( 2 w_{14}-\gamma (w_{62} + w_{68}) )   \\
    \Psi(e^1)_7 &= {\scriptstyle \frac{A}{1260}} (21(34w_{15} + 5 w_{29} - 3 w_{36}) \\ &+ 6\gamma (15  w_{56} - 12 w_{65} - 22  w_{69}) \\& +(36 \gamma^2 + 194 a_1 - 230 a_2) w_{83}  )
    \\
    \Psi(e^1)_8 &= {\scriptstyle \frac{A}{60}} ( 23 w_{13} + 2 w_{17} + 5 w_{35}-2 \gamma w_{63} ) 
    \end{align*}
    \begin{align*} 
    \\
    \Psi(e^1)_9 &= -{\scriptstyle \frac{A}{1260}} ( 21(34 w_{16} + 5 w_{30} - 3 w_{37}) \\ &-6\gamma ( 15  w_{55} - 22  w_{64} + 12 w_{70})\\ & +(36 \gamma^2 - 230 a_1  + 194 a_2) w_{87} ) 
     \\
    \Psi(e^1)_{10} &= {\scriptstyle \frac{A}{2520}} (1428 (2w_{28}+w_{31}+w_{38})  \\ &+312\gamma (2w_{71}-w_{74}) \\& +3(8 \gamma^2-367 a_1+359 a_2)w_{84}\\
    &-3(8\gamma^2+359 a_1-367 a_2) w_{89})  \\&-{\scriptstyle \frac{1531A^2}{1260}}  w_{107}
    \end{align*}
    \end{multicols}

    \begin{multicols}{2}
    \begin{align*}
       \Psi(e^1)_{13} &= {\scriptstyle \frac{A}{24}} (w_{62} - w_{68})  \\
    \Psi(e^1)_{14} &= -{\scriptstyle \frac{A}{15}}  w_{63} \\
    \Psi(e^1)_{15} &= {\scriptstyle \frac{A}{1680}} ( 555 w_{55} + 509 w_{64} - 599 w_{70} \\&+44 \gamma w_{87})  \\
    \Psi(e^1)_{16} &= {\scriptstyle \frac{A}{1680}} (- 555 w_{56} - 599 w_{65} - 509 w_{69} \\&+44 \gamma w_{83} )  \\
    \Psi(e^1)_{17} &= -{\scriptstyle \frac{A}{120}} (5 w_{62} + 3 w_{68})  \\
    \Psi(e^1)_{18} &= -{\scriptstyle \frac{A}{30}}  w_{63} \\
    \Psi(e^1)_{27} &= {\scriptstyle \frac{A}{168}} (9(-57w_{58}+36 w_{71}-17w_{74}) \\&+8 \gamma (7 w_{84}+10 w_{89}+2 w_{107}))  \\
    \Psi(e^1)_{28} &= {\scriptstyle \frac{A}{5040}} (3(150 w_{57} - 511 w_{61} - 1136 w_{72} \\& + 641 w_{73} + 1239 w_{77})-12\gamma (83 w_{85} \\&- 157  w_{88})-2(60 \gamma^2 - 101 a_1  \\&+ 41 a_2) w_{111} )  \\
    \Psi(e^1)_{29} &= {\scriptstyle \frac{A}{1680}} ( 7(- 3 w_{55} - 9 w_{64} + 11 w_{70})\\&+4 \gamma w_{87})  \\
    \Psi(e^1)_{30} &= {\scriptstyle \frac{A}{1680}} (7( 3 w_{56} + 11 w_{65} + 9 w_{69}) \\& +4 \gamma w_{83}) \\
    \Psi(e^1)_{31} &= 
     {\scriptstyle\frac{A}{168}}(47(w_{61}-4w_{72}-w_{73}+w_{77})\\&+22\gamma(w_{85}+w_{88})) -{\scriptstyle \frac{A^2}{1008}}71  w_{111}
    \\
    \Psi(e^1)_{32} &= 
    {\scriptstyle \frac{A}{14}}(9(2w_{71}-w_{74})-2\gamma(w_{84}-w_{89}))
    \\
    \Psi(e^1)_{33} &= -{\scriptstyle \frac{A}{504}} 
    (12(39 w_{59} - 47( w_{75} +  w_{79}) \\&- \gamma (22  w_{90} + 29  w_{103} - 27  w_{109})) \\&+ (96 \gamma^2  + 23 a_1  - 119 a_2) w_{112})
    \\
    \Psi(e^1)_{34} &= -{\scriptstyle \frac{A}{504}} 
    (12(
- 39 w_{60} + 47 (w_{76} -  w_{78}) \\&- \gamma(22  w_{86} + 29  w_{99} - 27  w_{108})) \\&-(96 \gamma^2  - 119 a_1  + 23 a_2) w_{113}
)\\
    \Psi(e^1)_{35} &= {\scriptstyle \frac{A}{40}} (w_{62} - w_{68}) \\
    \Psi(e^1)_{36} &= {\scriptstyle \frac{A}{1680}} (  7(9 w_{55} + 11 w_{64} - 17 w_{70}) \\&-4 \gamma w_{87}) 
    \\
    \Psi(e^1)_{37} &= -{\scriptstyle \frac{A}{1680}} (  7(9 w_{56} + 17 w_{65} + 11 w_{69}) \\&+ 4 \gamma w_{83}) \\
    \Psi(e^1)_{38} &= 
   {\scriptstyle \frac{A}{840}} (49(-w_{61}+4w_{72}+w_{73}-w_{77}) \\& -38\gamma (w_{85}+ w_{88})) +{\scriptstyle \frac{71A^2 }{5040}}  w_{111}
   \\
    \Psi(e^1)_{55} &= {\scriptstyle \frac{A}{30}}  w_{83} \\
    \Psi(e^1)_{56} &= -{\scriptstyle \frac{A}{30}}  w_{87} \\
    \Psi(e^1)_{57} &= -{\scriptstyle \frac{19A}{420}} (w_{84} - w_{89})
\\
    \Psi(e^1)_{58} &= {\scriptstyle \frac{A}{84}} (w_{85} + w_{88}) 
    \end{align*}
    \begin{align*} 
\\
    \Psi(e^1)_{59} &= {\scriptstyle \frac{A}{28}}  ( w_{86} - 11 w_{108} + 2 \gamma w_{113} ) \\
     \Psi(e^1)_{60} &= {\scriptstyle \frac{A}{28}} (- w_{90} + 11 w_{109} +2 \gamma w_{112} ) \\
    \Psi(e^1)_{61} &= -{\scriptstyle \frac{A}{20}} (w_{84} - w_{89}) \\
    \Psi(e^1)_{64} &= {\scriptstyle \frac{53 A}{840}} w_{83} \\
    \Psi(e^1)_{65} &= -{\scriptstyle \frac{53A}{840}}  w_{87} \\
    \Psi(e^1)_{69} &= -{\scriptstyle \frac{53 A}{840}}  w_{87} \\
    \Psi(e^1)_{70} &= -{\scriptstyle \frac{53 A}{840}}  w_{83} \\
    \Psi(e^1)_{71} &= -{\scriptstyle \frac{A}{84}} ( - 5 w_{85} + 34 w_{88}+\gamma w_{111}) \\
    \Psi(e^1)_{72} &= -{\scriptstyle \frac{A}{840}} (67 w_{84} - 237 w_{89} + 90 w_{107}) \\
    \Psi(e^1)_{73} &= {\scriptstyle \frac{23A}{84}} (w_{84} - w_{89}) \\
    \Psi(e^1)_{74} &= -{\scriptstyle \frac{29  A}{84}} (w_{85} + w_{88}) \\
    \Psi(e^1)_{75} &= -{\scriptstyle \frac{A}{84}} (  101 w_{86} + 28 w_{99}- 33 w_{108} \\&+20 \gamma w_{113}) \\
    \Psi(e^1)_{76} &= -{\scriptstyle \frac{A}{84}} (  - 101 w_{90}- 28 w_{103} + 33 w_{109} \\&+20 \gamma w_{112}) \\
    \Psi(e^1)_{77} &= -{\scriptstyle \frac{2A}{5}} (w_{84} - w_{89}) \\
    \Psi(e^1)_{78} &= -{\scriptstyle \frac{A}{6}} (-8 w_{90} + 2 w_{103} + 5 w_{109} + \gamma w_{112}) \\
\Psi(e^1)_{79} &= {\scriptstyle \frac{A}{6}} (8 w_{86} - 2 w_{99} - 5 w_{108} + \gamma w_{113}) \\
\Psi(e^1)_{80} &= -{\scriptstyle \frac{A}{6}} (11 w_{100} + 7 w_{105} + \gamma w_{116}) \\
\Psi(e^1)_{81} &= -{\scriptstyle \frac{A}{2}} (  - 3 (w_{101} - w_{104})+\gamma(w_{115} + w_{118})  ) \\
\Psi(e^1)_{82} &= -{\scriptstyle \frac{A}{6}} (-7 w_{100} - 11 w_{105} + \gamma w_{116}) \\
    \Psi(e^1)_{84} &= -{\scriptstyle \frac{ A}{12}} w_{111} \\
    \Psi(e^1)_{86} &= {\scriptstyle \frac{ A}{12}}  w_{112} \\
    \Psi(e^1)_{89} &= {\scriptstyle \frac{ A}{12}}  w_{111} \\
    \Psi(e^1)_{90} &= {\scriptstyle -\frac{ A}{12}}  w_{113} \\
     \Psi(e^1)_{99} &= {\scriptstyle -\frac{ A}{6}}  w_{112} \\
    \Psi(e^1)_{100} &= {\scriptstyle -\frac{ A}{6}} (2 w_{115} + w_{118})  \\
    \Psi(e^1)_{101} &= {\scriptstyle -\frac{ A}{6}} w_{116}  \\
      \Psi(e^1)_{102} &= {\scriptstyle -\frac{A}{2}} w_{117} \\
    \Psi(e^1)_{102} &= {\scriptstyle -\frac{A}{2}} w_{117}  \\
    \Psi(e^1)_{103} &= {\scriptstyle \frac{A}{6}}  w_{113} \\
    \Psi(e^1)_{104} &= {\scriptstyle \frac{A}{6}} w_{116}  \\  
\Psi(e^1)_{105} &= {\scriptstyle \frac{A}{6}} (w_{115} + 2 w_{118})  \\
    \Psi(e^1)_{106} &= {\scriptstyle \frac{A}{2}} w_{119}  \\
    \Psi(e^1)_{108} &= {\scriptstyle \frac{A}{12}}  w_{112} \\
    \Psi(e^1)_{109} &= {\scriptstyle -\frac{A}{12}}  w_{113} \\
    \Psi(e^1)_{110} &= {\scriptstyle -\frac{A}{12}} (w_{115} - w_{118})  
\end{align*}
\end{multicols}

\begin{multicols}{2}
\begin{align*}
\Psi(e^2)_{5} &= -{\scriptstyle \frac{19A}{60}} w_{11} \\
\Psi(e^2)_{6} &= -{\scriptstyle \frac{7A}{30}} w_{12} \\
\Psi(e^2)_{7} &= 
{\scriptstyle \frac{A}{420}}(-98w_{13}+7(w_{17}+11w_{35}) \\& +8\gamma w_{63}) 
\\
\Psi(e^2)_{8} &= {\scriptstyle \frac{7A}{20}} w_{11} \\
\Psi(e^2)_{9} &= 
{\scriptstyle \frac{A}{420}}(7(22w_{14}-5w_{18}) \\&-2\gamma( w_{62}-9 w_{68}))
\\
\Psi(e^2)_{10} &= 
{\scriptstyle \frac{A}{2520}}(84(17 w_{15}-5 w_{29}-14 w_{36}) \\&-24\gamma(15 w_{56}+ w_{65}+9 w_{69} )\\&-(168\gamma^2-325a_1 +157a_2 )w_{83})
\\
\Psi(e^2)_{15} &= 
{\scriptstyle \frac{A}{1680}}(579 w_{62}+529 w_{68})
\\
\Psi(e^2)_{16} &= -{\scriptstyle \frac{11 A}{840}} w_{63} \\
\Psi(e^2)_{27} &= 
{\scriptstyle \frac{A}{210}}(3(49 w_{56}-43 w_{65})-127 w_{69} \\&-64\gamma w_{83})
\\
\Psi(e^2)_{28} &= 
{\scriptstyle \frac{A}{1680}}(-363w_{55}+787w_{64}-785 w_{70}\\&-672\gamma w_{87})
\\
\Psi(e^2)_{29} &= {\scriptstyle \frac{A}{240}}(w_{62} - 5w_{68}) \\
\Psi(e^2)_{30} &= {\scriptstyle \frac{A}{40}} w_{63} \\
\Psi(e^2)_{31} &= {\scriptstyle \frac{A}{1680}}(-513w_{55}+557 w_{64}-495 w_{70} \\&+168\gamma w_{87})
\\
\Psi(e^2)_{32} &= {\scriptstyle \frac{A}{840}}(267 w_{56}-223 w_{65}-261 w_{69}\\&+100\gamma w_{83})
\\
\Psi(e^2)_{33} &= {\scriptstyle \frac{A}{1680}}(3(570w_{57}-203w_{61}-248w_{72})\\&-1141w_{73}-739w_{77}+24\gamma(13 w_{85}\\&-7 w_{88})+(280\gamma ^2-69a_1-211a_2)w_{111})\\
\Psi(e^2)_{34} &= {\scriptstyle \frac{A}{168}}(15 w_{58}-52w_{71}+23w_{74}\\&+4\gamma(25 w_{84}+21 w_{89}+52w_{107}) )
\\
\Psi(e^2)_{36} &=  {\scriptstyle \frac{A}{240}}(17w_{62}+11w_{68})
\\
\Psi(e^2)_{37} &= {\scriptstyle \frac{A}{40}} w_{63} \\
\Psi(e^2)_{38} &= {\scriptstyle \frac{A}{1680}}(7(-9w_{55}+17w_{64}-11w_{70}) \\& +80\gamma w_{87})
\\
\Psi(e^2)_{57} &= {\scriptstyle \frac{A}{84}} w_{83} 
\\
\Psi(e^2)_{58} &= -{\scriptstyle \frac{19A}{420}} w_{87} \\
\Psi(e^2)_{59} &= {\scriptstyle \frac{A}{420}}(-4w_{84}+19w_{89}+165w_{107})
\\
\Psi(e^2)_{60} &= {\scriptstyle \frac{A}{84}}(-w_{85}+2w_{88}+6\gamma w_{111})
\\
\Psi(e^2)_{61} &= {\scriptstyle \frac{A}{20}} w_{83} \\
\Psi(e^2)_{71} &= {\scriptstyle \frac{41A}{120}} w_{87} \\
\Psi(e^2)_{72} &= {\scriptstyle \frac{A}{60}} w_{83}
\end{align*}
\begin{align*}
 \\
\Psi(e^2)_{73} &= -{\scriptstyle \frac{2A}{5}} w_{83} \\
\Psi(e^2)_{74} &= {\scriptstyle \frac{29A}{84}} w_{87} \\
\Psi(e^2)_{75} &= {\scriptstyle \frac{A}{840}}(823w_{84}+577w_{89}+700w_{107})
\\
\Psi(e^2)_{76} &= {\scriptstyle \frac{A}{12}}(-5w_{85}+7w_{88}+3\gamma w_{111}) 
\\
\Psi(e^2)_{77} &= {\scriptstyle \frac{23A}{84}} w_{83} \\
\Psi(e^2)_{78} &= {\scriptstyle \frac{A}{84}}(- 52w_{85} + 77w_{88}+13\gamma w_{111} ) \\
\Psi(e^2)_{79} &= {\scriptstyle \frac{A}{840}}(437w_{84} + 293w_{89} + 330w_{107}) \\
\Psi(e^2)_{80} &= {\scriptstyle -\frac{A}{42}}(   - 25w_{86} - 77w_{99}+16w_{108} \\& +13\gamma w_{113}) \\
\Psi(e^2)_{81} &= {\scriptstyle \frac{A}{84}}(  3( - 12w_{90}+42w_{103}) + 23w_{109} \\& +50\gamma w_{112}) \\
\Psi(e^2)_{82} &= {\scriptstyle \frac{A}{6}}( - 2w_{86} - 7w_{99} + w_{108} +3\gamma w_{113} ) \\
\Psi(e^2)_{99} &= {\scriptstyle \frac{A}{6}} w_{111} \\
\Psi(e^2)_{100} &= {\scriptstyle \frac{5A}{12}} w_{112} \\
\Psi(e^2)_{101} &= {\scriptstyle \frac{A}{4}} w_{113} \\
\Psi(e^2)_{102} &= {\scriptstyle -\frac{A}{6}}(2w_{115} + w_{118}) \\
\Psi(e^2)_{104} &= {\scriptstyle \frac{A}{6}} w_{113} \\
\Psi(e^2)_{105} &= -{\scriptstyle \frac{A}{6}} w_{112} \\
\Psi(e^2)_{108} &= -{\scriptstyle \frac{A}{12}} w_{111} \\
\Psi(e^2)_{110} &= {\scriptstyle \frac{A}{12}} w_{112} \\
 & \\
\Psi(e^3)_5 &= {\scriptstyle -\frac{7A}{20}} w_{12} \\
\Psi(e^3)_6 &= {\scriptstyle \frac{7A}{30}} w_{11} \\
\Psi(e^3)_7 &= {\scriptstyle \frac{A}{420}}( 119w_{14} + 35w_{18} \\& +2\gamma(9 w_{62} -  w_{68}) ) \\
\Psi(e^3)_8 &= {\scriptstyle \frac{19A}{60}} w_{12} \\
\Psi(e^3)_9 &= {\scriptstyle \frac{A}{420}}(  7(15w_{13} + w_{17} - 12w_{35}) \\&+8\gamma w_{63})
\\
\Psi(e^3)_{10} &= {\scriptstyle \frac{A}{2520}}(  84(- 17w_{16} + 5w_{30} + 14w_{37}) \\&- 24\gamma(15 w_{55} - 9 w_{64} -  w_{70})\\& +(168\gamma^2  + 157a_1 - 325a_2)w_{87}) \\
\Psi(e^3)_{15} &= {\scriptstyle \frac{11A}{840}} w_{63} 
\\
\Psi(e^3)_{16} &= {\scriptstyle -\frac{A}{1680}}(529w_{62} + 579w_{68}) \\
\Psi(e^3)_{27} &= {\scriptstyle -\frac{A}{840}}(  321w_{55} - 247w_{64} + 293w_{70}\\&+356\gamma w_{87}) \\
\Psi(e^3)_{28} &= {\scriptstyle \frac{A}{1680}}( - 939w_{56} + 1357w_{65} + 1463w_{69}\\&+424\gamma w_{83}) \\
\Psi(e^3)_{29} &= {\scriptstyle -\frac{A}{40}} w_{63}
\end{align*}
\end{multicols}
\begin{multicols}{2}
\begin{align*}
\Psi(e^3)_{30} &= {\scriptstyle \frac{A}{240}}(5w_{62} - w_{68}) \\
\Psi(e^3)_{31} &= {\scriptstyle \frac{A}{1680}}(  513w_{56} - 495w_{65} - 557w_{69}\\&+168\gamma w_{83}) \\
\Psi(e^3)_{32} &= {\scriptstyle -\frac{A}{840}}( 3(- 89w_{55} + 87w_{64}) - 223w_{70}\\&+100\gamma w_{87}) 
\\
\Psi(e^3)_{33} &= {\scriptstyle \frac{A}{168}}(  15w_{58} + 52w_{71} - 29w_{74}  \\&+ 4\gamma(21 w_{84} + 25 w_{89} +52 w_{107})) \\
\Psi(e^3)_{34} &= {\scriptstyle \frac{A}{1680}}(  3(- 570w_{57} + 327w_{61} -248w_{72}) \\&+ 769w_{73} + 1111w_{77} - 24\gamma(7 w_{85} \\&- 13 w_{88}) -(280\gamma^2 - 211a_1- 69a_2)w_{111}) \\
\Psi(e^3)_{36} &= {\scriptstyle -\frac{A}{40}} w_{63} \\
\Psi(e^3)_{37} &= {\scriptstyle -\frac{A}{240}}(11w_{62} + 17w_{68}) \\
\Psi(e^3)_{38} &= {\scriptstyle \frac{A}{1680}}(  7(9w_{56} - 11w_{65} - 17w_{69})\\& +80\gamma w_{83}) \\
\Psi(e^3)_{57} &= {\scriptstyle -\frac{A}{84}} w_{87} \\
\Psi(e^3)_{58} &= {\scriptstyle -\frac{19A}{420}} w_{83} \\
\Psi(e^3)_{59} &= {\scriptstyle -\frac{A}{84}}( - 2w_{85} + w_{88}+6\gamma w_{111}) \\
\Psi(e^3)_{60} &= {\scriptstyle -\frac{A}{420}}(19w_{84} - 4w_{89} + 165w_{107}) \\
\Psi(e^3)_{61} &= {\scriptstyle -\frac{A}{20}} w_{87} \\
\Psi(e^3)_{71} &= {\scriptstyle \frac{A}{280}} w_{83} \\
\Psi(e^3)_{72} &= {\scriptstyle -\frac{29A}{84}} w_{87} \\
\Psi(e^3)_{73} &= {\scriptstyle \frac{2A}{5}} w_{87} \\
\Psi(e^3)_{74} &= {\scriptstyle \frac{29A}{84}} w_{83} \\
\Psi(e^3)_{75} &= {\scriptstyle \frac{A}{12}}( - 7w_{85} + 5w_{88}+3\gamma w_{111}) \\
\Psi(e^3)_{76} &= {\scriptstyle \frac{A}{840}}(577w_{84} + 823w_{89} + 700w_{107}) \\
\Psi(e^3)_{77} &= {\scriptstyle -\frac{23A}{84}} w_{87} \\
\Psi(e^3)_{78} &= {\scriptstyle \frac{A}{840}}(293w_{84} + 437w_{89} + 330w_{107}) \\
\Psi(e^3)_{79} &= {\scriptstyle -\frac{A}{84}}( - 77w_{85} + 52w_{88}+13\gamma w_{111}) 
\\
\Psi(e^3)_{80} &= {\scriptstyle \frac{A}{6}}(   2w_{90}+ 7w_{103} - w_{109}+3\gamma w_{112}) \\
\Psi(e^3)_{81} &= {\scriptstyle \frac{A}{84}}(   3(13w_{86} -42w_{99})- 23w_{108} \\& +50\gamma w_{113}) 
\\
\Psi(e^3)_{82} &= {\scriptstyle -\frac{A}{42}}( 25w_{90} + 77w_{103} - 16w_{109}\\&+13\gamma w_{112} )
\end{align*}
\begin{align*} 
\\
\Psi(e^3)_{100} &= {\scriptstyle \frac{A}{6}} w_{113} \\
\Psi(e^3)_{101} &= {\scriptstyle -\frac{A}{6}} w_{112} \\
% \Psi(e^3)_{102} &= 0
\Psi(e^3)_{103} &= -{\scriptstyle\frac{A}{6}} w_{111} \\
\Psi(e^3)_{104} &= -{\scriptstyle\frac{A}{4}} w_{112} \\
\Psi(e^3)_{105} &= -{\scriptstyle\frac{5A}{12}} w_{113} \\
\Psi(e^3)_{106} &= {\scriptstyle\frac{A}{6}}(w_{115} + 2w_{118}) \\
\Psi(e^3)_{109} &= {\scriptstyle\frac{A}{12}} w_{111} \\
\Psi(e^3)_{110} &= -{\scriptstyle\frac{A}{12}} w_{113} 
\\ & \\
\Psi(e^4)_7 &= {\scriptstyle =\frac{A}{20}} w_{11} \\
%\Psi(e^4)_8 &= 0 \\
\Psi(e^4)_9 &= {\scriptstyle \frac{A}{20}} w_{12} \\
\Psi(e^4)_{10} &= {\scriptstyle \frac{A}{70}}(7( w_{13} + 2w_{17} - w_{35})-4\gamma w_{63} ) \\
\Psi(e^4)_{27} &= {\scriptstyle -\frac{3A}{140}} w_{63} \\
\Psi(e^4)_{28} &= {\scriptstyle -\frac{A}{840}}(167w_{62} + 149w_{68}) \\
\Psi(e^4)_{31} &= {\scriptstyle \frac{3A}{140}}(w_{62} - w_{68}) \\
\Psi(e^4)_{32} &= {\scriptstyle -\frac{3A}{70}} w_{63} \\
\Psi(e^4)_{33} &= {\scriptstyle \frac{A}{840}}(- 129w_{55} + 131w_{64} - 185w_{70} \\& +44\gamma w_{87} ) \\
\Psi(e^4)_{34} &= {\scriptstyle \frac{A}{840}}(  129w_{56} - 185w_{65} - 131w_{69} \\& +44\gamma w_{83}) \\
\Psi(e^4)_{59} &= {\scriptstyle -\frac{3A}{140}} w_{83} \\
\Psi(e^4)_{60} &= {\scriptstyle \frac{3A}{140}} w_{87} \\
\Psi(e^4)_{75} &= {\scriptstyle -\frac{197A}{840}} w_{83} \\
\Psi(e^4)_{76} &= {\scriptstyle \frac{197A}{840}} w_{87} \\
\Psi(e^4)_{78} &= {\scriptstyle \frac{193A}{840}} w_{87} \\
\Psi(e^4)_{79} &= {\scriptstyle \frac{193A}{840}} w_{83} \\
\Psi(e^4)_{80} &= {\scriptstyle -\frac{A}{420}}(197 w_{84} + 193 w_{89} + 230 w_{107}) \\
\Psi(e^4)_{81} &= {\scriptstyle \frac{A}{84}}( 39(- w_{85} + w_{88})+16\gamma w_{111}) \\
\Psi(e^4)_{82} &= {\scriptstyle \frac{A}{420}}(193w_{84} + 197w_{89} + 230w_{107}) \\
\Psi(e^4)_{100} &= {\scriptstyle \frac{A}{12}} w_{111} \\
% \Psi(e^4)_{101} &= 0 \\
\Psi(e^4)_{102} &= {\scriptstyle- \frac{A}{12}} w_{112} \\
% \Psi(e^4)_{103} &= 0 \\
% \Psi(e^4)_{104} &= 0 \\
\Psi(e^4)_{105} &= {\scriptstyle- \frac{A}{12}} w_{111} \\
\Psi(e^4)_{106} &= {\scriptstyle \frac{A}{12}} w_{113} 
\end{align*}
\end{multicols}

Let us now write down the results of the third step from Section \ref{ts-cec}, starting with the generic solutions. Moreover, this allows us to show that generically all the conformal Killing tensors are reducible.

\begin{prop}\label{generic-CKT}
If  $a_1\neq a_2$,  then generically, there is $27$-dimensional space of reducible conformal Killing tensors.
\end{prop}
\begin{proof}
The generic $\Phi$-invariant subspace $\mathbb{S}^{gen}$ annihilated by the curvature $R^{\Phi}$ of $\Phi$ can be solved directly; see Maple file \cite{bgg-comp}. The resulting subset $\mathbb{S}^{gen}\subset S^2 \mathfrak{so}(2,4)$ provides a $27$-dimensional space of conformal Killing tensors. We have the following $27$ free variables
\begin{gather*} w_{77}, w_{78}, w_{79,} w_{80}, w_{81}, w_{82}, w_{86}, w_{89}, w_{90},w_{100}, w_{101}, w_{102}, w_{104}, w_{105}, w_{106},  w_{108}, \\ w_{109}, w_{110}, w_{111}, w_{112}, w_{113}, w_{115}, w_{116}, w_{117}, w_{118}, w_{119}, w_{120}
\end{gather*}
and the following dependent, where missing variables are either vanishing or are determined by the equations for $\mathbb{V}:$
\begin{multicols}{2}
\begin{align*}
w_{10}&={\scriptstyle \frac{1}{84}}(23a_1^2+38a_1a_2+23a_2^2)w_{111}\\
w_{27}&=-{\scriptstyle \frac{1}{84}}(19a_1+9a_2)\gamma w_{111}\\
w_{28}&={\scriptstyle \frac{1}{21}}(4a_1+17a_2)w_{89}\\
w_{31}&=-{\scriptstyle \frac12} (a_1+a_2) w_{89} \\
w_{32}&=-{\scriptstyle \frac13} \gamma(a_1+a_2) w_{111}\\
w_{33}&= {\scriptstyle \frac{1}{84}}a_1(465w_{86}-355w_{108}+202\gamma w_{113})\\
&+{\scriptstyle \frac{1}{84}}a_2(291 w_{86}-233 w_{108}+134\gamma w_{113}) \\ 
w_{34}&= {\scriptstyle \frac{1}{84}}a_1(291w_{90}-233w_{109}-134\gamma w_{112})\\ 
&-{\scriptstyle \frac{1}{84}}a_2(-465 w_{90} +355 w_{109}+ 202\gamma w_{112}) \\
w_{38}&=-{\scriptstyle \frac12} (a_1+a_2) w_{89}\\ 
 w_{57}&=-{\scriptstyle\frac{10}{3}}w_{77}+  {\scriptstyle \frac16}(12\gamma^2-a_1-a_2)w_{111} \\
 w_{58}&={\scriptstyle \frac43}\gamma w_{89} 
 \\
w_{59} &= -{\scriptstyle \frac83}w_{79} +{\scriptstyle \frac{1}{18}}\gamma(-324w_{90}+252w_{109})  \\
&+   {\scriptstyle \frac{1}{18}}(144\gamma^2-a_1-5a_2)w_{112}
\\
 w_{60} &= {\scriptstyle \frac83}w_{78} + {\scriptstyle \frac{1}{18}}\gamma(324w_{86}-252w_{108}) \\
 &+  {\scriptstyle \frac{1}{18}}(144\gamma^2-5a_1-a_2)w_{113}
 \end{align*}
\begin{align*}
\\
 w_{61}&=-{\scriptstyle \frac72}w_{77}+{\scriptstyle \frac12}(4\gamma^2-a_1-a_2)w_{111}  
\\
 w_{71}&={\scriptstyle \frac23}\gamma w_{89}\\
 w_{72}&={\scriptstyle \frac14}w_{77}  +{\scriptstyle \frac{1}{24}}(a_1-a_2)w_{111}\\
 w_{73}&= -{\scriptstyle \frac72}w_{77}+  {\scriptstyle \frac12}(4\gamma^2-a_1-a_2)w_{111}\\
 w_{74}&={\scriptstyle \frac43}\gamma w_{89} \\
 w_{75}&=-3w_{79}+ {\scriptstyle \frac13}\gamma(-54w_{90}+42w_{109}) \\
 &+{\scriptstyle \frac13}(24\gamma^2-a_1-2a_2)w_{112} \\
 w_{76} &= 3w_{78}+{\scriptstyle \frac13}\gamma(54w_{86}-42w_{108})
 \\&+{\scriptstyle \frac13}(24\gamma^2-2a_1-a_2)w_{113} \\
w_{84}&=w_{89} \\
w_{85}&={\scriptstyle \frac13}\gamma w_{111} \\
w_{88}&=-{\scriptstyle \frac13}\gamma w_{111}\\
 w_{99}&= 4w_{86}-3w_{108} + 2\gamma w_{113}\\ 
w_{103} &= 4w_{90} -3w_{109} -2\gamma w_{112}\\
w_{107}&=-w_{89} 
\end{align*}
\end{multicols}

Comparing this with Proposition \ref{reducible-prop}, we see that all of these are reducible for dimensional and continuity reasons.
\end{proof}

In Proposition \ref{generic-CKT}, we exhausted all reducible conformal Killing tensors. Therefore, the additional conformal Killing tensors that may appear for special values of the parameters are irreducible. To finish the proof of Theorem \ref{main-CKT}, we restrict ourselves to the complement $C$ of $\mathbb{S}^{gen}$ in $\mathbb{V}$.  From now on, we parametrize $\mathbb{V}$ by variables $v_i$ to distinguish the general and special case, and we define the complement $C$ by the following equations:
\begin{align*}
v_{72}&=0,  \:\: v_{78}=0, \:  v_{79}=0,\: \:  v_{80}=0, \:  v_{81}=0,  \: \: v_{82}=0, \:  v_{86}=0, \: \:  v_{89}=0, \: v_{90}=0,
\: v_{100}=0, \\ v_{101}&=0, v_{102}=0, v_{104}=0, v_{105}=0, v_{106}=0, v_{108}=0, v_{109}=0, v_{110}=0, 
v_{111}=0, v_{112}=0, \\
v_{113}&=0, v_{115}=0, v_{116}=0, v_{117}=0, v_{118}=0, v_{119}=0, v_{120}=0.
\end{align*}
We first find the subspace of $C$ annihilated by the curvature $R^{\Phi}$ of the prolongation connection $\Phi$ restricted to $C$, i.e., the subspace $C_0\subset C$ such that $R^{\Phi}(C_0)\subset \mathbb{S}^{gen}$. As this is an overdetermined system of equations, which, except for one distinguished equation, can be solved independently of the values of parameters, we present only the final result.

\newpage
\begin{lem}
If  $a_1\neq a_2$, then the subspace $C_0$ of $C$ such that $R^{\Phi}(C_0)\subset \mathbb{S}^{gen}$ is given by

\begin{multicols}{2}
\begin{align*}
v_5 &= {\scriptstyle \frac{1}{840}} (225 \gamma^2 + 86 a_1 - 26 a_2) v_{68} \\
v_7 &= {\scriptstyle \frac{1}{9}} \gamma v_{37} + {\scriptstyle \frac{1}{840}} (-165 \gamma^2 + 223 a_1 - 463 a_2) v_{64} \\
&+ {\scriptstyle \frac{1}{840}} (-615 \gamma^2 + 107 a_1 - 467 a_2) v_{70} \\
&+ {\scriptstyle \frac{1}{5040}} \gamma (-2335 \gamma^2 + 869 a_1 - 829 a_2) v_{87} \\
v_8 &= {\scriptstyle \frac{1}{840}} (225 \gamma^2 - 26 a_1 + 86 a_2) v_{68} \\
v_9 &= -{\scriptstyle \frac{1}{9}} \gamma v_{36} + {\scriptstyle \frac{1}{840}} (615 \gamma^2 + 467 a_1 - 107 a_2) v_{65} \\
&+ {\scriptstyle \frac{1}{840}} (-165 \gamma^2 - 463 a_1 + 223 a_2) v_{69} \\
&+ {\scriptstyle \frac{1}{5040}} \gamma (2335 \gamma^2 + 829 a_1 - 869 a_2) v_{83} \\
v_{13} &= {\scriptstyle \frac{10}{9}} v_{35} \\
v_{14} &= -{\scriptstyle \frac{19}{18}} \gamma v_{68} \\
v_{15} &= {\scriptstyle \frac{1}{36}}(40 v_{36}-\gamma(61 v_{65}-23  v_{69}))\\
&+ {\scriptstyle \frac{1}{2520}} (-2765 \gamma^2 + 219 a_1 + 341 a_2) v_{83} \\
v_{16} &= {\scriptstyle \frac{1}{36}}(40v_{37}-\gamma (23 v_{64}-61  v_{70}))\\
&+ {\scriptstyle \frac{1}{2520}} (-2765 \gamma^2 + 341 a_1 + 219 a_2) v_{87} \\
v_{17} &= -{\scriptstyle \frac{1}{18}} v_{35} \\
v_{18} &= -{\scriptstyle \frac{19}{36}} \gamma v_{68} 
\\ 
v_{29} &={\scriptstyle \frac{1}{72}}(4v_{36}+ \gamma(79 v_{65}-41 v_{69}))\\
&+ {\scriptstyle \frac{1}{720}} (415 \gamma^2 + 47 a_1 + 33 a_2) v_{83} 
\end{align*}
\begin{align*}
\\
v_{30} &={\scriptstyle \frac{1}{72}}(4v_{37}+\gamma(41 v_{64}+79 v_{70})) \\
&+ {\scriptstyle \frac{1}{720}} (415 \gamma^2 + 33 a_1 + 47 a_2) v_{87} \\
v_{31} &=   {\scriptstyle \frac{1}{18}}(22v_{38}+9\gamma v_{74})\\
&+ {\scriptstyle \frac{1}{630}} (-35 \gamma^2 - 137 a_1 - 143 a_2) v_{84} \\ &
+{\scriptstyle \frac{1}{9}} (5 \gamma^2 - 5 a_1 - 5 a_2) v_{107}\\
v_{32} &= {\scriptstyle \frac{1}{18}}\gamma (26v_{61}-16v_{73}+35v_{77})\\ 
&+ {\scriptstyle \frac{1}{252}} (-567 \gamma^2 - 11 a_1 + 11 a_2) v_{85} \\
&+ {\scriptstyle \frac{1}{252}} (567 \gamma^2 - 11 a_1 + 11 a_2) v_{88} \\
v_{55} &=  {\scriptstyle \frac{1}{2}} (4 v_{64} + 4 v_{70}+ 3\gamma v_{87})  \\
v_{56} &=  {\scriptstyle \frac{1}{2}} (- 4 v_{65} + 4 v_{69} -3 \gamma v_{83})  \\
v_{57} &={\scriptstyle \frac{1}{6}}(- 2 v_{61} + 6(v_{73} - v_{77})\\
&+9( \gamma v_{85} -  \gamma v_{88})) \\
v_{58} &= {\scriptstyle \frac{1}{6}}( 12 v_{74} -  \gamma(9 v_{84}  +10  v_{107})) \\
v_{59} &=  {\scriptstyle \frac{1}{6}} ( 4 v_{75} -9 \gamma v_{103}) \\
v_{60} &= {\scriptstyle \frac{1}{6}}( 4 v_{76} + 9 \gamma v_{99} )\\
v_{62} &= v_{68} 
\end{align*}
\end{multicols}
together with the distinguished condition $\gamma v_{35} = 0$,
and the following equal to zero
\begin{gather*}
    v_{1 }, v_{2 }, v_{3 }, v_{6 }, v_{11 }, v_{12 }, v_{63 }.
\end{gather*}
\end{lem}

Considering the distinguished equation for $C_0$, we need to branch into two cases: 
\begin{enumerate}
\item $\epsilon=\gamma=0,a_2=\pm 2+a_1$ (symmetric spaces), and 
\item $\epsilon=0,\gamma=1, v_{35}=0$.
\end{enumerate}

In the case of symmetric spaces $\epsilon=\gamma=0,a_2=\pm 2+a_1$, there is only one parameter $a_1$, and the linear equations for the maximal $\Phi$-invariant subspace of $C_0+\mathbb{S}^{gen}$ can be directly solved including the parameter $a_1$. Generically, there are no additional solutions. However, for $a_1 = \frac23$ (isometric to $a_1 = \frac83,a_1 = \frac23$) and $a_1 = -\frac83$, there are $9$-dimensional spaces of the additional solutions. There is no problem with exponentiation, because we have fixed all the parameters, and we get the irreducible conformal Killing tensors from Proposition \ref{expicit-CKT}.

The case $\epsilon=0,\gamma=1, v_{35}=0$ is more complicated, but it is easy to compute the subspace $C_1$ of $C_0$ such that $\Phi(\fk)C_1\subset  C_0+\mathbb{S}^{gen}.$

\begin{lem}
If  $a_1\neq a_2,\epsilon=0,\gamma=1, v_{35}=0$,  then the subspace $C_1$ of $C_0$ is given by
\begin{align*}
v_{10 }&= -{\scriptstyle \frac{1}{840}}( -100a_1-56 a_2+21) v_{73}+{\scriptstyle \frac{1}{80}}(-26a_1-26a_2+7) v_{77}
\\
&+{\scriptstyle \frac{1}{16800}}(3440a_1^2+16184a_1a_2+14504a_2^2+10766a_1+1978a_2-4347) v_{85}\\
&+{\scriptstyle \frac{1}{16800}}(-12040a_1^2-17944a_1a_2-4144a_2^2-5146a_1-7598a_2+4347) v_{88}
\\
v_{27 }&= {\scriptstyle \frac{1}{4}} v_{77}+{\scriptstyle \frac{1}{840}}(272a_1-1028a_2-189) v_{85}+{\scriptstyle \frac{1}{840}}(242a_1+682a_2+21) v_{88} \\
v_{28 }&= 
-{\scriptstyle \frac{1}{(a_1-a_2)}}{\scriptstyle \frac{1}{1680}} (-336 a_1-1176 a_2-2058) v_{74}\\
&-{\scriptstyle \frac{1}{(a_1-a_2)}}{\scriptstyle \frac{1}{1680}} (-302 a_1^2-908 a_1 a_2+1210 a_2^2-1778 a_1+518 a_2+525) v_{84}\\
&-{\scriptstyle \frac{1}{(a_1-a_2)}}{\scriptstyle \frac{1}{1680}} (380 a_1^2-2440 a_1 a_2+2060 a_2^2-3388 a_1-1148 a_2-1694) v_{107}
\end{align*}
\begin{align*}
\\
v_{33 }&= -{\scriptstyle \frac{1}{4}} v_{76}+{\scriptstyle \frac{1}{40}} (44a_1+44a_2+47) v_{99} 
\\
v_{34 }&= {\scriptstyle \frac{1}{4}}  v_{75}+ {\scriptstyle \frac{1}{40}} (44a_1+44a_2+47) v_{103}  
\\
v_{38 }&= 
-{\scriptstyle \frac{1}{40}}(23a_1+5a_2+17) v_{84}-{\scriptstyle \frac{1}{20}}(4a_1+4a_2+17) v_{107}
\\
v_{71 }&={\scriptstyle \frac{1}{(a_1-a_2)}}{\scriptstyle \frac{1}{240}} (228 a_1-12 a_2+294) v_{74}+ {\scriptstyle \frac{1}{(a_1-a_2)}}{\scriptstyle \frac{1}{240}} (110 a_1+70 a_2-75) v_{84}\\ 
&+ {\scriptstyle \frac{1}{(a_1-a_2)}}{\scriptstyle \frac{1}{240}} (324 a_1+324 a_2+242) v_{107}
\end{align*}
and the following equal to zero
\begin{gather*}
v_{36 },
v_{37 },
v_{64 },
v_{65 },
v_{68 },
v_{69 },
v_{70 },
v_{83 },
v_{87 }.
\end{gather*}
\end{lem}

To proceed further, we consider the subspace $C_{2,3}$ on $C_1$ such that $\Phi(\fk)C_{2,3}\subset  C_1+\mathbb{S}^{gen}$, $\Phi(\fk)\Phi(\fk)C_{2,3}\subset  C_1+\mathbb{S}^{gen}$ and $\Phi(\fk)\Phi(\fk)\Phi(\fk)C_{2,3}\subset  C_1+\mathbb{S}^{gen}.$ Considering a simpler system of linear equations would include some branches for the values of parameters that do not contain any solutions. This system can be separated into two subsystems:

\begin{enumerate}
\item The first subsystem depends on $v_{74},v_{84},v_{107}$ and does not admit a nontrivial solution for any values of the parameters, that is, $$v_{74}=0,\quad v_{84}=0,\quad v_{107}=0.$$

\item The second subsystem depends on the remaining variables and there is a distinguished case $a_1=-1,a_3=3 $ with solution  $$v_{75} = -{\scriptstyle \frac72} v_{103},\quad v_{73} = 18v_{85},\quad v_{77} = -6v_{85},\quad v_{88} = 5v_{85}.$$ In such a case, there is no problem with exponentiation, because we have fixed all the parameters and get the irreducible conformal Killing tensors from Proposition \ref{expicit-CKT}.
\end{enumerate}

\section{Killing tensors on conformally flat plane waves - Proof of Theorem \ref{main-fKT}}\label{sec-fKT}

We are in a situation  $a:=a_1=a_2$ and $\gamma$ does not appear in the definition of the metric. If $a=0$, then we are in the flat case and there are $50$ reducible Killing tensors. The following Proposition provides the formulae for the irreducible Killing tensors from Theorem \ref{main-fKT}. It is a direct computation to check that these tensors are Killing tensors. To show that they are irreducible and that we have all of them, we show that they come from a complement of the space of the reducible Killing tensors we described in Proposition \ref{reducible-prop}. As the corresponding conformal Killing tensors are all reducible, we express them as multiples of the (inverse) metric plus the corresponding reducible conformal Killing tensor.

\begin{prop}\label{expl-fKT}
The homogeneous plane waves $g_{0,\pm 1,\pm 1,0}$ have one-dimensional spaces of irreducible Killing tensors
$$(z_1\partial_{z_1}+z_2\partial_{z_2}+2x_-\partial_{x_-})\partial_{x_-}-x_+g_{0,\pm 1,\pm 1,0}^{-1}.$$
The homogeneous plane wave $g_{1,\frac34,\frac34,0}$ has one-dimensional space of irreducible Killing tensors 
\begin{gather*}
\Big(z_1\partial_{z_1}+
z_2\partial_{z_2}+4x_-\partial_{x_-}
-2x_+\partial_{x_+}\Big)\Big(\frac{z_1\partial_{z_1}+z_2\partial_{z_2}-2x_+\partial_{x_+}}{x_+^2}+\frac{z_1^2+z_2^2}{x_+^3}\partial_{x_-}\Big)\\
-\Big(\frac{z_1^2+z_2^2}{x_+^2}-\frac{4x_-}{x_+}\Big)g_{1,\frac34,\frac34,0}^{-1}.
\end{gather*}

The homogeneous plane wave $g_{1,-\frac3{16},-\frac3{16},0}$ has a six-dimensional space of irreducible Killing tensors
\begin{gather*}
\Big( c_3 (x_+^{\frac14}\partial_{z_1} -\frac{z_1}{4x_+^{\frac34}}\partial_{x_+})
-2(c_1-c_5)(x_+^{\frac14}\partial_{z_2} -\frac{z_2}{4x_+^{\frac34}}\partial_{x_+})\Big)\Big(x_+^{\frac14}z_2\partial_{x_-}+\frac{z_1z_2}{4x_+^{\frac34}}\partial_{z_1}
\\
-(x_-x_+^\frac14+\frac{z_1^2-z_2^2}{8x_+^{\frac34}})\partial_{z_2}+z_2(\frac{x_-}{4x_+^{\frac34}}+\frac{3(z_1^2+z_2^2)}{32x_+^{\frac74}})\partial_{x_+}\Big)+\Big( 2c_1 (x_+^{\frac14}\partial_{z_1} -\frac{z_1}{4x_+^{\frac34}}\partial_{x_+})
\\
+c_3(x_+^{\frac14}\partial_{z_2} -\frac{z_2}{4x_+^{\frac34}}\partial_{x_+})\Big)\Big(x_+^{\frac14}z_1\partial_{x_-}-(x_-x_+^\frac14+\frac{-z_1^2+z_2^2}{8x_+^{\frac34}})\partial_{z_1}+\frac{z_1z_2}{4x_+^{\frac34}}\partial_{z_2}
\\
+z_1(\frac{x_-}{4x_+^{\frac34}}+\frac{3(z_1^2+z_2^2)}{32x_+^{\frac74}})\partial_{x_+}\Big) 
+c_4\Big(x_+^{\frac12}\partial_{x_-}+\frac{z_1}{4x_+^{\frac12}}\partial_{z_1}+\frac{z_2}{4x_+^{\frac12}}\partial_{z_2}+\frac{z_1^2+z_2^2}{16x_+^{\frac32}}\partial_{x_+}\Big)^2
\\
+c_2\Big(x_+^{\frac12}\partial_{x_-}+\frac{z_1}{4x_+^{\frac12}}\partial_{z_1}+\frac{z_2}{4x_+^{\frac12}}\partial_{z_2}+\frac{z_1^2+z_2^2}{16x_+^{\frac32}}\partial_{x_+}\Big)\Big(x_+^{\frac14}\partial_{z_1} -\frac{z_1}{4x_+^{\frac34}}\partial_{x_+}\Big)
\\
+c_6\Big(x_+^{\frac12}\partial_{x_-}+\frac{z_1}{4x_+^{\frac12}}\partial_{z_1}+\frac{z_2}{4x_+^{\frac12}}\partial_{z_2}+\frac{z_1^2+z_2^2}{16x_+^{\frac32}}\partial_{x_+}\Big)\Big(x_+^{\frac14}\partial_{z_2} -\frac{z_2}{4x_+^{\frac34}}\partial_{x_+}\Big)
\\
-\Big( \frac{c_1 (z_1^2-z_2^2) }{4x_+^{\frac12}}+\frac{    c_2  z_1}{4x_+^\frac14}+\frac{ c_3 z_1 z_2 }{4x_+}+ c_4 (\frac{ x_- x_+^{\frac12} }{2}+\frac{  z_1^2+ z_2^2 }{2x_+^{\frac12}})+\frac{c_5 z_2^2 }{4x_+^{\frac12}}+\frac{  c_6  z_2}{4x_+^\frac14}\Big)g_{1,-\frac3{16},-\frac3{16},0}^{-1}.
\end{gather*}
\end{prop}

We know from \cite{GL,bgg-my} that the representation $\rho: \mathfrak{sl}(5,\mathbb{R})\to \mathfrak{gl}(\mathbb{V})$ is the Cartan component $\mathbb{V}\subset S^2 \Lambda^2 (\mathbb{R}^5)^*.$ The space $S^2 \Lambda^2 (\mathbb{R}^5)^*$ is of dimension $55$,  corresponds to positions $(i,j,k,l)$ for $i<j,k<l,(i,j)\leq (k,l)$, and is embedded in $\otimes^4 (\mathbb{R}^5)^*$ via the symmetries of the tensors in $S^2 \Lambda^2 (\mathbb{R}^5)^*$. Let us recall that $S^2 \Lambda^2 (\mathbb{R}^5)^*$ decomposes into distinguished eigenspaces for the $\rho$-action of the grading element in $\mathfrak{sl}(5,\mathbb{R})$, \cite{parabook}, with eigenvalues shifted to $0,\dots,2$ according our conventions. We order the positions $(i,j,k,l)$ in the following way that describes this grading explicitly:

\begin{enumerate}
\item Firstly, we order 
\begin{itemize}
    \item[(i)]
$(1,j,1,l)$, $1<j\leq l$ with variables $w_1,\dots,w_{10}$
\end{itemize}
corresponding to the highest grading component $\mathbb{X}=S^2 \Lambda^2 (\mathbb{R}^5)^*_0$. The projection $\pi: S^2 \Lambda^2 (\mathbb{R}^5)^*\to S^2\fc$ is $\pi(v)=v_{1,j,1,l}e^{j-1}\otimes e^{l-1}.$
\item Next, we order 
\begin{itemize}
\item[(ii)] $(1,j,k,l)$, $1<j,k,l$ with variables $w_{11},\dots,w_{34}$
\end{itemize}
corresponding to the grading component $S^2 \Lambda^2 (\mathbb{R}^5)^*_1$.
\item Finally, we order 
\begin{itemize}
\item[(iii)] $(i,j,k,l)$, $1<i,j,k,l$ with variables $w_{35},\dots,w_{55}$
\end{itemize}
corresponding to the grading component $S^2 \Lambda^2 (\mathbb{R}^5)^*_2$.
\end{enumerate}

The Cartan component $\mathbb{V}$ is $50$-dimensional, and we describe it (using the Casimir element) as the subspace of $S^2 \Lambda^2 (\mathbb{R}^5)^*$ with the following $50$ free variables 
$$w_1 \dots, w_{21}, w_{24},\dots, w_{28}, w_{31},\dots, w_{43},w_{44},\dots, w_{55}$$
and the dependent:
\begin{multicols}{3}
\begin{align*}
w_{22}&=-w_{32}+w_{27}\\
w_{23}&=w_{18}-w_{14}
\end{align*}
\begin{align*}\\
w_{29}&=w_{19}-w_{15} \\
w_{30}&=w_{25}-w_{w16}
\end{align*}
\begin{align*}\\
w_{44}&=w_{47}+w_{40}
\end{align*}
\end{multicols}
As described in Section \ref{ss-sec}, we find the prolongation connection $\Phi.$ The representation $\rho$ is induced by the standard action on $\otimes^4 (\mathbb{R}^5)^*$ and the map $\beta$ from Proposition \ref{pnormal}. The components $\Psi$ of the prolongation connection are as follows (a direct check that these formulae are correct can also be found in the Maple file \cite{bgg-comp} accompanying the arXiv submission):

\begin{multicols}{2}
\begin{align*}
\Psi(e^1)_{11} &= -{\scriptstyle \frac{1}{3}} aw_2  \\
\Psi(e^1)_{12} &= -{\scriptstyle \frac{1}{3}} aw_3  
\\
\Psi(e^1)_{13} &= {\scriptstyle \frac{2}{3}} aw_4 
\\
\Psi(e^1)_{15} &= {\scriptstyle \frac{1}{3}} a w_7 
\\
\Psi(e^1)_{16} &= {\scriptstyle \frac{1}{3}} a w_9 
\\
\Psi(e^1)_{17} &= {\scriptstyle \frac{1}{6}} a \left( 3 w_4 - w_5 \right) \\
\Psi(e^1)_{18} &= -{\scriptstyle \frac{1}{6}} a w_6  \\
\Psi(e^1)_{19} &= {\scriptstyle \frac{1}{4}} a w_7 \\
\Psi(e^1)_{20} &= -{\scriptstyle \frac{1}{4}} a w_9 \\
\Psi(e^1)_{21} &= -{\scriptstyle \frac{1}{4}} a w_{10} \\
\Psi(e^1)_{24} &= {\scriptstyle \frac{1}{6}} a \left( 3 w_4 - w_8 \right) \\
\Psi(e^1)_{25} &= {\scriptstyle \frac{1}{4}} a w_9 \\
\Psi(e^1)_{26} &= {\scriptstyle \frac{1}{4}} a w_7 \\
\Psi(e^1)_{28} &= -{\scriptstyle \frac{1}{4}} a w_{10} \\
\Psi(e^1)_{31} &= {\scriptstyle \frac{1}{3}} a w_{10} \\
\Psi(e^1)_{35} &= -a \left(  \epsilon (w_4 - {\scriptstyle \frac{1}{3}} w_5)- {\scriptstyle \frac{3}{2}} w_{13} + w_{17} \right) \\
\Psi(e^1)_{36} &= {\scriptstyle \frac{1}{6}} a \left( 2 \epsilon w_6 + 3 w_{14} - 6 w_{18} \right) \\
\Psi(e^1)_{37} &= -{\scriptstyle \frac{1}{12}}a(2\epsilon w_9+21w_{16}-6w_{25})\\
%-{\scriptstyle \frac{1}{6}} a \left( \epsilon w_7 + {\scriptstyle \frac{21}{2}} w_{15} - 3 w_{19} \right) \\
\Psi(e^1)_{38} &= {\scriptstyle \frac{1}{12}} a \left( 2 \epsilon w_9 - 11 w_{16} - 4 w_{20} - 2 w_{25} \right) \\
\Psi(e^1)_{39} &= {\scriptstyle \frac{1}{6}} a \left( \epsilon w_{10} - 4 w_{21} - w_{31} \right) \\
\Psi(e^1)_{40} &= -{\scriptstyle \frac{2}{3}} a w_{27} \\
\Psi(e^1)_{41} &= -a \left(  \epsilon (w_4 - {\scriptstyle \frac{1}{3}} w_8) - {\scriptstyle \frac{3}{2}} w_{13} + w_{24} \right) \\
\Psi(e^1)_{42} &=  -{\scriptstyle \frac{1}{12}}a(2\epsilon w_9+21w_{16}-6w_{25})\\
%-{\scriptstyle \frac{1}{6}} a \left( \epsilon w_9 + {\scriptstyle \frac{21}{2}} w_{16} - 3 w_{25} \right) \\
\Psi(e^1)_{43} &= -{\scriptstyle \frac{1}{12}} a \left( 2\epsilon w_7 - 11 w_{15} - 2w_{19} + 4 w_{26} \right) \\
\Psi(e^1)_{45} &= {\scriptstyle \frac{1}{6}} a \left( \epsilon w_{10} - 4 w_{28} - w_{31} \right) \\
\Psi(e^1)_{46} &= -{\scriptstyle \frac{2}{3}} a \left( \epsilon w_{10} - 3 w_{31} \right) \\
\Psi(e^1)_{47} &= {\scriptstyle \frac{2}{3}} a w_{32} \\
\Psi(e^1)_{48} &= a w_{33} \\
\Psi(e^1)_{49} &= a w_{34} \\
\Psi(e^1)_{50} &= {\scriptstyle \frac{2}{3}} a \left( w_{21} + w_{28} \right) \\
\Psi(e^1)_{51} &= {\scriptstyle \frac{1}{3}} a w_{34} \\
\Psi(e^1)_{52} &= -{\scriptstyle \frac{1}{3}} a w_{33} 
\\ & \\
\Psi(e^2)_{11} &= {\scriptstyle \frac{1}{6}} a \left( 3 w_4 - w_5 \right) \\
\Psi(e^2)_{12} &= -{\scriptstyle \frac{1}{6}} w_6 a \\
\Psi(e^2)_{13} &= {\scriptstyle \frac{1}{12}} a w_7 \\
\Psi(e^2)_{14} &= -{\scriptstyle \frac{1}{12}} a w_9 
\end{align*}
\begin{align*} \\
\Psi(e^2)_{15} &= -{\scriptstyle \frac{1}{12}} a w_{10} \\
\Psi(e^2)_{17} &= a w_7 
\\
\Psi(e^2)_{18} &= {\scriptstyle \frac{1}{3}} a w_9 \\
\Psi(e^2)_{19} &= {\scriptstyle \frac{1}{3}} a w_{10} \\
\Psi(e^2)_{24} &= {\scriptstyle \frac{1}{4}} a w_7 \\
\Psi(e^2)_{35} &= -{\scriptstyle \frac{1}{2}} a \left( 4 \epsilon w_7 + 6 w_{15} - 9 w_{19} \right) \\
\Psi(e^2)_{36} &= -{\scriptstyle \frac{1}{6}} a \left( 4 \epsilon w_9 + 8 w_{16} + w_{20} - 10 w_{25} \right) \\
\Psi(e^2)_{37} &= -{\scriptstyle \frac{1}{12}} a \left( 8 \epsilon w_{10} + 7 w_{21} - 20 w_{31}\right) \\
\Psi(e^2)_{38} &= -{\scriptstyle \frac{3}{4}} a \left( w_{27} - 3 w_{32} \right) \\
\Psi(e^2)_{39} &= {\scriptstyle \frac{3}{2}} a w_{33} \\
\Psi(e^2)_{40} &= a w_{34} \\
\Psi(e^2)_{41} &= -{\scriptstyle \frac{1}{6}} a \left( 4 \epsilon w_7 + 2 w_{15} - 7 w_{19} + 2 w_{26} \right) \\
\Psi(e^2)_{42} &= -{\scriptstyle \frac{1}{12}} a \left( 7 w_{27} - 5 w_{32} \right) \\
\Psi(e^2)_{43} &= {\scriptstyle \frac{1}{12}} a \left( 5 w_{21} - 4 w_{28} \right) \\
\Psi(e^2)_{45} &= {\scriptstyle \frac{1}{6}} a w_{33} \\
\Psi(e^2)_{47} &= -{\scriptstyle \frac{2}{3}} a w_{34} \\
&\\
\Psi(e^3)_{11} &= -{\scriptstyle \frac{1}{6}}a w_6  \\
\Psi(e^3)_{12} &= {\scriptstyle \frac{1}{6}} a \left( 3 w_4 - w_8 \right) \\
\Psi(e^3)_{13} &= {\scriptstyle \frac{1}{12}} a w_9 \\
\Psi(e^3)_{14} &= {\scriptstyle \frac{1}{12}} a w_7 \\
\Psi(e^3)_{16} &= -{\scriptstyle \frac{1}{12}} a w_{10} \\
\Psi(e^3)_{17} &= {\scriptstyle \frac{1}{4}} a w_9 \\
\Psi(e^3)_{18} &= {\scriptstyle \frac{5}{12}} a w_7 \\
\Psi(e^3)_{24} &= a w_9 \\
\Psi(e^3)_{25} &= {\scriptstyle \frac{1}{3}} a w_{10} \\
\Psi(e^3)_{35} &= -{\scriptstyle \frac{1}{6}} a \left( 4 \epsilon w_9 + 2 w_{16} - 2 w_{20} - 7 w_{25} \right) \\
\Psi(e^3)_{36} &= -{\scriptstyle \frac{1}{6}} a \left( 4 \epsilon w_7 + 8 w_{15} - 10 w_{19} - w_{26} \right) \\
\Psi(e^3)_{37} &= -{\scriptstyle \frac{1}{12}} a \left( 7 w_{27} - 2 w_{32} \right) \\
\Psi(e^3)_{38} &= {\scriptstyle \frac{1}{12}} a \left( 4 w_{21} - 5 w_{28} \right) \\
\Psi(e^3)_{39} &= {\scriptstyle \frac{1}{6}} a w_{34} \\
\Psi(e^3)_{40} &= {\scriptstyle \frac{1}{3}} a w_{33} \\
\Psi(e^3)_{41} &= -{\scriptstyle \frac{1}{2}} a \left( 4 \epsilon w_9 + 6 w_{16} - 9 w_{25} \right) \\
\Psi(e^3)_{42} &= -{\scriptstyle \frac{1}{12}} a \left( 8 \epsilon w_{10} + 7 w_{28} - 20 w_{31} \right) \\
\Psi(e^3)_{43} &= {\scriptstyle \frac{3}{4}} a \left( w_{27} + 2 w_{32} \right) \\
\Psi(e^3)_{45} &= {\scriptstyle \frac{3}{2}} a w_{34} \\
\Psi(e^3)_{47} &= {\scriptstyle \frac{2}{3}} a w_{33} 
\end{align*}
\end{multicols}
\begin{multicols}{2}
\begin{align*}
\Psi(e^4)_{11} &= {\scriptstyle \frac{1}{12}} a w_7 \\
\Psi(e^4)_{12} &= {\scriptstyle \frac{1}{12}} a w_9 \\
\Psi(e^4)_{13} &= {\scriptstyle \frac{1}{3}} a w_{10} 
\\
\Psi(e^4)_{17} &= {\scriptstyle \frac{1}{4}} a w_{10} \\
\Psi(e^4)_{24} &= {\scriptstyle \frac{1}{4}} a w_{10} \\
\Psi(e^4)_{35} &= -{\scriptstyle \frac{1}{6}} a \left( 4 \epsilon w_{10} + 2 w_{21} - 7 w_{31} \right) 
\end{align*}
\begin{align*}\\
\Psi(e^4)_{36} &= -{\scriptstyle \frac{1}{6}} a \left( 2 w_{27} - w_{32} \right) \\
\Psi(e^4)_{37} &= -{\scriptstyle \frac{3}{4}} a w_{33} 
\\
\Psi(e^4)_{38} &= -{\scriptstyle \frac{7}{12}} a w_{34} \\
\Psi(e^4)_{41} &= -{\scriptstyle \frac{1}{6}} a \left( 4 \epsilon w_{10} + 2 w_{28} - 7 w_{31} \right) \\
\Psi(e^4)_{42} &= -{\scriptstyle \frac{3}{4}} a w_{34} \\
\Psi(e^4)_{43} &= {\scriptstyle \frac{7}{12}} a w_{33} 
\end{align*}
\end{multicols}
Let us now write down the results of the third step from Section \ref{ts-cec}, starting with the generic solutions. Moreover, this allows us to show that generically, most of the Killing tensors are reducible.

\begin{prop}\label{generic-fKT}
If  $a\neq 0$,  then, generically, 
\begin{enumerate}
\item there is $28$-dimensional space of Killing tensors, which are all reducible for $\epsilon\neq 0$, or 
\item there are irreducible Killing tensors from Proposition \ref{expl-fKT} for $\epsilon= 0$.
\end{enumerate}
\end{prop}
\begin{proof}
The generic $\Phi$-invariant subspace $\mathbb{S}^{gen}$ annihilated by the curvature $R^{\Phi}$ of $\Phi$ can be directly computed, see the Maple file \cite{bgg-comp}. The resulting subset $\mathbb{S}^{gen}\subset \mathbb{V}$ provides a $ 28$-dimensional space of  Killing tensors. We have the following $28$ free variables
\begin{gather*} w_{1},\dots,w_{14},w_{17},\dots,w_{20},w_{24},\dots,w_{27},w_{35},w_{36},w_{38},w_{41},w_{43},w_{50}
\end{gather*}
and the following dependent, where missing variables are either vanishing or are determined by the equations for $\mathbb{V}:$
\begin{multicols}{3}
\begin{align*}
w_{15}&=\epsilon w_7-w_{19}\\
w_{16}&=\epsilon w_9-w_{25}\\
w_{31}&=\epsilon w_{10}\\
w_{32}&=2 w_{27}
\end{align*}
\begin{align*}
\\
w_{37}&=-2\epsilon^2w_7+3\epsilon w_{19}-{\scriptstyle \frac{1}{12}}aw_7\\
w_{39}&=-{\scriptstyle \frac14} aw_{10}\\
w_{40}&=-\epsilon w_{27}\\
w_{42}&=-2\epsilon^2w_9+3\epsilon w_{25}-{\scriptstyle \frac{1}{12}}aw_9
\end{align*}
\begin{align*}
\\
w_{45}&=-{\scriptstyle \frac14}aw_{10}\\
w_{46}&=-{\scriptstyle \frac13}(-3\epsilon^2+a)w_{10}\\
w_{47}&=2\epsilon w_{27}
\end{align*}
\end{multicols}

Comparing this with Proposition \ref{reducible-prop}, we see that these are reducible for dimensional and continuity reasons in the case $\epsilon\neq 0$. Otherwise, there are irreducible Killing tensors from Proposition \ref{expl-fKT} for $\epsilon= 0$.
\end{proof}

In Proposition \ref{generic-fKT}, we exhausted all reducible Killing tensors. Therefore, the additional Killing tensors that may appear for special values of the parameters are irreducible. To finish the proof of Theorem \ref{main-fKT}, we restrict ourselves to complement $C$ of $\mathbb{S}^{gen}$ in $\mathbb{V}$. %parametrized by $v_i$ variables and
From now on, we parametrize $\mathbb{V}$ by variables $v_i$, and we define the complement $C$ by the following equations:
\begin{align*}
v_{1}&=0,\dots,v_{14}=0,v_{17}=0,\dots,v_{20}=0,v_{24}=0,\dots,v_{27}=0,\\
v_{35}&=0,v_{36}=0,v_{38}=0,v_{41}=0,v_{43}=0,v_{50}=0.
\end{align*}
We first find the subspace of $C$ annihilated by the curvature $R^{\Phi}$ of the prolongation connection $\Phi$ restricted to $C$, i.e., the subspace $C_0\subset C$ such that $R^{\Phi}(C_0)\subset \mathbb{S}^{gen}$. As this is an overdetermined system of equations which, except for two distinguished equations, can be solved independently of the parameter values, we present only the final result.

\begin{lem}
If  $a\neq 0$,  then the subspace $C_0$ of $C$ such that $R^{\Phi}(C_0)\subset \mathbb{S}^{gen}$ is given by
\begin{multicols}{2}
\begin{align*}
    v_{37} &= -{\scriptstyle \frac{5}{4}}\epsilon v_{15} \\
    v_{39} &= {\scriptstyle \frac{1}{2}}\epsilon v_{21} - {\scriptstyle \frac{7}{4}}\epsilon v_{31} + v_{46} 
    \end{align*}
    \begin{align*} \\
    v_{40} &= -{\scriptstyle \frac{3}{4}}\epsilon v_{32} 
   \\
    v_{42} &= -{\scriptstyle \frac{5}{4}}\epsilon v_{16}
     \end{align*}
     \end{multicols}
     \begin{multicols}{2}
     \begin{align*}
    v_{45} &= {\scriptstyle \frac{1}{2}}\epsilon v_{28}  - {\scriptstyle \frac{7}{4}}\epsilon v_{31} + v_{46} \\
    v_{47} &= \epsilon v_{32} 
    \end{align*}
    \begin{align*}
      \\
    v_{48} &=  \epsilon v_{33} \\
    v_{49} &= \epsilon  v_{34} 
\end{align*}
\end{multicols}
together with the distinguished condition $\epsilon v_{33}=0, \epsilon v_{34} = 0$,
and the following equal to zero
\begin{gather*}
    v_{51 }, v_{52 }, v_{53 }, v_{54 }, v_{55 }.
\end{gather*}
\end{lem}
Considering the distinguished equation for $C_0$, we need to branch into two cases:
\begin{enumerate}
\item $\epsilon=0,a=\pm 1$ (symmetric spaces), and
\item $\epsilon=1, v_{33}=0, v_{34}=0.$
\end{enumerate}

In the case of symmetric spaces $\epsilon=0,a=\pm 1$, there are no more parameters and the linear equations for the maximal $\Phi$-invariant subspace of $C_0+\mathbb{S}^{gen}$ can be solved directly and there are no additional solutions.

The case $\epsilon=1, v_{33}=0, v_{34}=0$ is more complicated, because we immediately obtain branching when we compute the subspace $C_1$ such that $\Phi(\fk)C_1\subset C_0+\mathbb{S}^{gen}.$

If $a\neq -\frac{3}{16}$, then $C_1$ is characterized by 

\begin{multicols}{3}
\begin{align*}
v_{15} &= 0 \\
v_{16} &= 0 
\end{align*}
\begin{align*} \\
v_{21} &= {\scriptstyle \frac{1}{2}} v_{31} 
\\
v_{28} &= {\scriptstyle \frac{1}{2}} v_{31} 
\end{align*}
\begin{align*} \\
v_{32} &= 0 \\
v_{46} &= \left( -{\scriptstyle \frac{16}{15}} a + {\scriptstyle \frac{9}{5}} \right) v_{31}
\end{align*}
\end{multicols}
\noindent
Then subspace $C_2$ of $C_1$ given by $\Phi(\fk)C_2\subset C_1+\mathbb{S}^{gen}$ consists of equations 
$$v_{31}(4a-3)=0,\  v_{31}(16a+3)(4a-3)=0,\  v_{31}(64a+27)(4a-3)=0,$$ and thus for $a\neq \frac{3}{4}$, there are no further solutions.
In the case $a= \frac{3}{4}$, there is one solution corresponding to the irreducible Killing tensor in Proposition \ref{expl-fKT}.

In the case $a= -\frac{3}{16}$, there are $6$ solutions corresponding to the irreducible Killing tensors in Proposition \ref{expl-fKT} given by $v_{46} = 2v_{31}$.

\section{Killing tensors on conformally non-flat homogeneous plane waves - Proof of Theorems \ref{main-KT} and \ref{main-KTCKT}}\label{sec-KT}

We are in the situation $a_1\neq a_2$. The following proposition provides the formulae for the irreducible Killing tensors from Theorem \ref{main-KT}. It is a direct computation to check that these tensors are Killing. To show that they are irreducible and that we have all of them, we show that they come from a complement of the space of the reducible Killing tensors we described in Proposition \ref{reducible-prop}. 
We also determine when they come from a complement of the space of reducible conformal Killing tensors we described in Proposition \ref{reducible-prop} to prove Theorem \ref{main-KTCKT}. In the case where the irreducible Killing tensors correspond to reducible conformal Killing tensors, we express them as multiples of the (inverse) metric plus the corresponding reducible conformal Killing tensor.

\begin{prop}\label{expl-KT}
The homogeneous plane waves $g_{0,a_1,a_2,\gamma}$ have one-dimensional spaces of irreducible Killing tensors 
$$c_1\big(z_1\partial_{z_1}+z_2\partial_{z_2}+2x_-\partial_{x_-}\big)\partial_{x_-}-c_1x_+g_{0,a_1,a_2,\gamma}^{-1}.$$

The homogeneous plane waves $g_{0,0,\pm 2,0}$ have two-dimensional spaces of irreducible Killing tensors 
$$\big(z_1\partial_{z_1}+z_2\partial_{z_2}+2x_-\partial_{x_-}\big)\big(c_1\partial_{z_1}+c_2\partial_{x_-}\big)- \big(c_1z_1+c_2x_+\big)g_{0,0,\pm 2,0}^{-1}.$$

The homogeneous plane waves $g_{0,\pm \frac83,\pm \frac23,0}$ have two-dimensional spaces of irreducible Killing tensors 
$$c_1\big(z_1\partial_{z_1}+z_2\partial_{z_2}+2x_-\partial_{x_-}\big)\partial_{x_-}-c_1x_+g_{0,\pm \frac83,\pm \frac23,0}^{-1}+c_2\big(z_2\partial_{z_1}\partial_{z_2}-z_1\partial_{z_2}^2\mp\frac23z_1z_2^2\partial_{x_-}^2\big).$$

The homogeneous plane waves $g_{1,0,a_2,0}$ have one-dimensional spaces of irreducible Killing tensors 

$$ c_1\big(z_1\partial_{z_1}+z_2\partial_{z_2}+2x_-\partial_{x_-}\big)\partial_{z_1}-c_1z_1g_{1,0,a_2,0}^{-1}.$$

The homogeneous plane wave $g_{1,0,\frac34,0}$ has a five-dimensional space of irreducible Killing tensors 

\begin{gather*}
c_1\big(z_1\partial_{z_1}+z_2\partial_{z_2}+2x_-\partial_{x_-}\big)\partial_{z_1}- c_1z_1g_{1,0,a_2,0}^{-1}\\
    + c_{3} \partial_{x_+}^2 -  \big({c_{4} z_{1} -  2c_{5}}\big) \partial_{x_+} \partial_{z_{1}} - \frac{ c_{3} z_{2}}{x_+} \partial_{x_+} \partial_{z_{2}} - \frac{ c_{3} z_{2}^2}{x_+^2} \partial_{x_+} \partial_{x_-} \\
    + \big(
    {c_{2} x_+^{\frac12}  z_{2} }+
    c_{4} x_-  -
    \frac{  c_{4} z_{2}^2}{4 x_+}\big) 
    \partial_{z_{1}}^2 -  \big(
    { c_{2} x_+^{\frac12}  z_{1} }
     - \frac{  c_{4}  z_{1} z_{2}}{2 x_+}
     +      \frac{   c_{5}  z_{2} }{ x_+}
        \big) \partial_{z_{1}} \partial_{z_{2}}\\
    -  \big(
        \frac{   c_{2} z_{1} z_{2} }{2 x_+^{\frac12}}
        -\frac{ c_{4}  z_{1} z_{2}^2 }{2 x_+^2}+
    \frac{   c_{5}  z_{2}^2 }{ x_+^2}
    \big) \partial_{z_{1}} \partial_{x_-} 
    + \big(
     \frac{   c_{3} x_-  }{ x_+}
    -\frac{c_{4}  z_{1}^2 }{4 x_+}
     + \frac{  c_{5} z_{1}}{ x_+}
    \big) \partial_{z_{2}}^2 \\
    + \big(
   \frac{ c_{2} z_{1}^2 }{2 x_+^{\frac12}}+
      \frac{  c_{3} x_-  z_{2} }{ x_+^2}+
               \frac{c_{3} z_{2}^3 }{4 x_+^3}
        -\frac{c_{4}  z_{1}^2 z_{2} }{4 x_+^2}
        +     \frac{   c_{5} z_{1} z_{2} }{ x_+^2}
    \big) \partial_{z_{2}} \partial_{x_-} \\
    + \big(
    \frac{   c_{2} z_{1}^2 z_{2} }{4 x_+^{\frac32}}+
    \frac{ 3 c_{3} z_{2}^4 }{16 x_+^4}+
    \frac{  c_{3} x_-  z_{2}^2 }{4 x_+^3}
    -\frac{ c_{4}  z_{1}^2 z_{2}^2 }{16 x_+^3}+
    \frac{  c_{5} z_{1} z_{2}^2}{4 x_+^3}
    \big) \partial_{x_-}^2.
\end{gather*}

The homogeneous plane wave $g_{1,0,-\frac3{16},0}$ has a six-dimensional space of irreducible Killing tensors 

\begin{gather*}
c_1z_1g_{1,0,a_2,0}^{-1}-c_1\big(z_1\partial_{z_1}+z_2\partial_{z_2}+2x_-\partial_{x_-}\big)\partial_{z_1}\\
 -2c_6x_+ \partial_{x_+}^2+\big(
  c_2x_+^{\frac34}
  -c_3x_+^{\frac34}z_1
 -c_5x_+^\frac12z_2
 -c_6z_2
 \big)\partial_{x_+}\partial_{z_2}\\
 + \big(
 -\frac{3c_2z_2}{4x_+^\frac14}
 +\frac{3c_3z_1z_2}{4x_+^\frac14}
 +\frac{c_5z_2^2}{2x_+^\frac12}
  +2c_6x_-\big)\partial_{x_+}\partial_{x_-}+
  \big(c_3x_+^\frac34x_--c_4x_+^\frac12z_2)\partial_{z_1}\partial_{z_2}
 \\ 
+ \big(-\frac{3c_3x_-z_2}{4x_+^\frac14}-\frac{c_3z_2^3}{16x_+^\frac54}+\frac{c_4z_2^2}{2x_+^\frac12} \big)\partial_{z_1}\partial_{x_-}
+ \big(c_4x_+^\frac12z_1+c_5x_+^\frac12x_-+c_6x_-\big) \partial_{z_2}^2\\
- \big(
\frac{c_4z_2z_1}{2x_+^\frac12}
+\frac{c_5z_2x_-}{2x_+^\frac12}
+\frac{c_5z_2^3}{16x_+^\frac32}
+\frac{c_6z_2^3}{16x_+^2}\big)\partial_{z_2}\partial_{x_-}
\\
+ \big(
-\frac{c_2z_2^3}{16x_+^\frac94}
+\frac{c_3z_2^3z_1}{16x_+^\frac94}
+\frac{c_4z_2^2z_1}{16x_+^\frac32}
+\frac{c_5z_2^2x_-}{16x_+^\frac32}
+\frac{3c_5z_2^4}{64 x_+^\frac52}+
\frac{3c_6z_2^2x_-}{16x_+^2}
+\frac{c_6z_2^4}{64x_+^3} \big)\partial_{x_-}^2.
\end{gather*}

The homogeneous plane waves $g_{1,4a_2+\frac34,a_2,0}$ have one-dimensional spaces of irreducible Killing tensors

\begin{gather*}
 c_1(2x_+^\frac12 z_2 \partial_{z_1}\partial_{z_2}-\frac{z_2^2}{x_+^\frac12} \partial_{z_1}\partial_{x_-} -2x_+^\frac12z_1 \partial_{z_2}^2+\frac{z_1z_2}{x_+^\frac12} \partial_{z_2}\partial_{x_-}- \frac{(4a_2+1)z_2^2z_1}{2x_+^\frac32}\partial_{x_-}^2.
\end{gather*}
\end{prop}

The description of $\mathbb{V}$ remains as in the previous section, and the prolongation connection has the following $\Psi$ (a direct check that these formulae are correct can also be found in the Maple file \cite{bgg-comp} accompanying the arXiv submission):
\begin{multicols}{2}
\begin{align*}
\Psi(e^1)_{11} &= -{\scriptstyle\frac{1}{3}}(2a_1 - a_2)w_2 \\
\Psi(e^1)_{12} &= {\scriptstyle\frac{1}{3}}(a_1 - 2a_2)w_3 \\
\Psi(e^1)_{13} &= {\scriptstyle\frac{1}{3}}(a_1 + a_2)w_4 \\
\Psi(e^1)_{14} &= {\scriptstyle\frac{1}{3}}(a_1 - a_2)w_6 
\end{align*}
\begin{align*}\\
\Psi(e^1)_{15} &= {\scriptstyle\frac{1}{3}}a_1w_7\\
\Psi(e^1)_{16} &= {\scriptstyle\frac{1}{3}}a_2w_9 \\
\Psi(e^1)_{17} &= {\scriptstyle\frac{1}{6}}a_1(3w_4 - 2w_5) + {\scriptstyle\frac{1}{6}}a_2w_5 \\
\Psi(e^1)_{18} &= {\scriptstyle\frac{1}{12}}(a_1 - 3a_2) w_6
\end{align*}
\end{multicols}
\begin{multicols}{2}
\begin{align*}
\Psi(e^1)_{19} &= {\scriptstyle\frac{1}{12}}(a_1 + 2a_2)w_7 \\
\Psi(e^1)_{20} &= -{\scriptstyle\frac{1}{4}}a_1w_9 \\
\Psi(e^1)_{21} &= -{\scriptstyle\frac{1}{4}}a_1w_{10} \\
\Psi(e^1)_{24} &= {\scriptstyle\frac{1}{6}}a_2(3w_4 - 2w_8) + {\scriptstyle\frac{1}{6}}a_1w_8 \\
\Psi(e^1)_{25} &= {\scriptstyle\frac{1}{12}}(2a_1 + a_2)w_9\\
\Psi(e^1)_{26} &= {\scriptstyle\frac{1}{4}}a_2w_7\\
\Psi(e^1)_{28} &= -{\scriptstyle\frac{1}{4}}a_2w_{10} \\
\Psi(e^1)_{31} &= {\scriptstyle\frac{1}{6}}(a_1 + a_2)w_{10} 
\\
\Psi(e^1)_{35} &= {\scriptstyle\frac{1}{6}}a_1(\epsilon(-6w_4 + 4w_5) - 3\gamma w_6 + 9w_{13} \\
&- 12w_{17}) - {\scriptstyle\frac{1}{6}}a_2(2\epsilon w_5 -3w_6\gamma - 6w_{17})  \\
\Psi(e^1)_{36} &= {\scriptstyle\frac{1}{4}}(a_1 - a_2)\gamma(2w_4 - w_5 - w_8) \\
&+ {\scriptstyle\frac{1}{12}}a_1(2\epsilon w_6 - 15w_{14} - 6w_{18}) \\
&+ {\scriptstyle\frac{1}{12}}a_2(2\epsilon w_6 + 21w_{14} - 6w_{18}) \\
\Psi(e^1)_{37} &= {\scriptstyle\frac{1}{12}}a_1(2\epsilon w_7 - 3\gamma w_9 - 15w_{15} - 6w_{19}) \\
&- {\scriptstyle\frac{1}{12}}a_2(4w_7\epsilon - 3w_9\gamma +6w_{15} - 12w_{19}) \\
\Psi(e^1)_{38} &= {\scriptstyle\frac{1}{12}}a_1(2\epsilon w_9 + \gamma w_7 - 11w_{16} \\
&- 12w_{20} - 2w_{25}) - {\scriptstyle\frac{1}{12}}a_2(\gamma w_7 - 8w_{20}) \\
\Psi(e^1)_{39} &= {\scriptstyle\frac{1}{6}}a_1(\epsilon w_{10} - 6w_{21} - w_{31}) + {\scriptstyle\frac{1}{3}}a_2w_{21} \\
\Psi(e^1)_{40} &= {\scriptstyle\frac{1}{12}}a_1(-\gamma w_{10} - 8w_{27} + 8w_{32}) \\
&+ {\scriptstyle\frac{1}{12}}a_2(\gamma w_{10} - 8w_{32}) \\
\Psi(e^1)_{41} &= {\scriptstyle\frac{1}{6}}a_2(\epsilon(-6w_4 + 4w_8) + 3\gamma w_6 + 9w_{13} \\
&- 12w_{24}) - {\scriptstyle\frac{1}{6}}a_1(3\gamma w_6 + 2\epsilon w_8 - 6w_{24}) \\
\Psi(e^1)_{42} &= {\scriptstyle\frac{1}{12}}a_1(-4\epsilon w_9 - 3\gamma w_7 - 6w_{16} + 12w_{25}) \\
&+ {\scriptstyle\frac{1}{12}}a_2(3w_7\gamma + 2\epsilon w_9 - 15w_{16} - 6w_{25}) \\
\Psi(e^1)_{43} &= {\scriptstyle\frac{1}{12}}(-2\epsilon w_7 + \gamma w_9 + 11w_{15} + 2w_{19} \\
&- 12w_{26})a_2 - {\scriptstyle\frac{1}{12}}a_1(\gamma w_9 - 8w_{26}) \\
\Psi(e^1)_{45} &= {\scriptstyle\frac{1}{6}}a_2(\epsilon w_{10} - 6w_{28} - w_{31}) + {\scriptstyle\frac{1}{3}}a_1w_{28} \\
\Psi(e^1)_{46} &= -{\scriptstyle\frac{1}{3}}(a_1 + a_2)(\epsilon w_{10} - 3w_{31}) \\
\Psi(e^1)_{47} &= {\scriptstyle\frac{2}{3}}(a_1 - a_2)w_{27} + {\scriptstyle\frac{2}{3}}a_2w_{32} \\
\Psi(e^1)_{48} &= {\scriptstyle\frac{1}{3}}(2a_1 + a_2)w_{33} \\
\Psi(e^1)_{49} &= {\scriptstyle\frac{1}{3}}(a_1 + 2a_2)w_{34} \\
\Psi(e^1)_{50} &= {\scriptstyle\frac{2}{3}}a_2w_{21} + {\scriptstyle\frac{2}{3}}a_1w_{28} \\
\Psi(e^1)_{51} &= {\scriptstyle\frac{1}{3}}a_1w_{34} \\
\Psi(e^1)_{52} &= -{\scriptstyle\frac{1}{3}}a_2w_{33} \\
 & \\
\Psi(e^2)_{11} &= {\scriptstyle \frac{1}{6}} a_1(3 w_4 - 2 w_5)  + {\scriptstyle \frac{1}{6}} a_2 w_5  \\
\Psi(e^2)_{12} &= -{\scriptstyle \frac{1}{12}}  (a_1 + a_2) w_6 \\
\Psi(e^2)_{13} &= -{\scriptstyle \frac{1}{12}}  (a_1 - 2 a_2) w_7 \\
\Psi(e^2)_{14} &= -{\scriptstyle \frac{1}{12}} a_1 w_9 
\end{align*}
\begin{align*}\\
\Psi(e^2)_{15} &= -{\scriptstyle \frac{1}{12}} a_1 w_{10} \\
\Psi(e^2)_{17} &= a_1 w_7 \\
\Psi(e^2)_{18} &= {\scriptstyle \frac{1}{3}} a_1 w_9 \\
\Psi(e^2)_{19} &= {\scriptstyle \frac{1}{3}} a_1 w_{10} \\
\Psi(e^2)_{24} &= {\scriptstyle\frac{1}{4}} a_2 w_7 \\
\Psi(e^2)_{35} &= -{\scriptstyle\frac{1}{2}} a_1 (4 \epsilon w_7 + 6 w_{15} - 9 w_{19})
\\
\Psi(e^2)_{36} &= {\scriptstyle \frac{1}{12}} a_1 (-8 \epsilon w_9 + 8 \gamma w_7 - 16 w_{16} - 9 w_{20}
 \\
&+ 20 w_{25}) - {\scriptstyle\frac{1}{12}} a_2 (8w_7 \gamma - 7 w_{20})   \\
\Psi(e^2)_{37} &={\scriptstyle \frac{1}{12}} a_1 (-8 \epsilon w_{10} - 9 w_{21} + 20 w_{31}) \\
&+ {\scriptstyle \frac{1}{6}} a_2 w_{21} \\
\Psi(e^2)_{38} &= -{\scriptstyle\frac{3}{4}} a_1 (w_{27} - 3 w_{32}) \\
\Psi(e^2)_{39} &= {\scriptstyle\frac{3}{2}} a_1 w_{33} \\
\Psi(e^2)_{40} &= a_1 w_{34} \\
\Psi(e^2)_{41} &= {\scriptstyle \frac{1}{6}}  a_2 (-4 \epsilon w_7 - 4 \gamma w_9 - 2 w_{15} + 7 w_{19}) \\
&+ {\scriptstyle\frac{1}{3}} a_1 (2w_9 \gamma - w_{26})  \\
\Psi(e^2)_{42} &= {\scriptstyle \frac{1}{12}} a_1(4 \gamma w_{10} - 2 w_{27} - 2 w_{32})  \\
&- {\scriptstyle \frac{1}{12}} a_2 (4\gamma w_{10} + 5 w_{27} - 7 w_{32}) \\
\Psi(e^2)_{43} &= {\scriptstyle\frac{5}{12}} a_2 w_{21} - {\scriptstyle \frac{1}{3}} a_1 w_{28} \\
\Psi(e^2)_{45} &= {\scriptstyle \frac{1}{6}} a_2 w_{33} \\
\Psi(e^2)_{46} &= -{\scriptstyle \frac{1}{3}}  (a_1 - a_2) w_{33} \\
\Psi(e^2)_{47} &= -{\scriptstyle\frac{2}{3}} a_1 w_{34} \\
& \\
\Psi(e^3)_{11} &= -{\scriptstyle \frac{1}{12}} \left( a_1 + a_2 \right)  w_6 \\
\Psi(e^3)_{12} &=  {\scriptstyle \frac{1}{6}}  a_2 \left( 3 w_4 - 2 w_8 \right)  + {\scriptstyle \frac{1}{6}} a_1 w_8  \\
\Psi(e^3)_{13} &= {\scriptstyle \frac{1}{12}}\left( 2 a_1 - a_2 \right) w_9  \\
\Psi(e^3)_{14} &= {\scriptstyle \frac{1}{12}} a_2 w_7 \\
% \Psi(e^3)_{15} &= 0 \\
\Psi(e^3)_{16} &= -{\scriptstyle \frac{1}{12}} a_2 w_{10} \\
\Psi(e^3)_{17} &= {\scriptstyle \frac{1}{4}} a_1 w_9 \\
\Psi(e^3)_{18} &= {\scriptstyle \frac{5}{12}} a_2 w_7 \\
\Psi(e^3)_{24} &= a_2 w_9 \\
\Psi(e^3)_{25} &= {\scriptstyle \frac{1}{3}} a_2 w_{10} \\
\Psi(e^3)_{35} &=  {\scriptstyle \frac{1}{6}} a_1 \left( -4 \epsilon w_9 + 4 \gamma w_7 - 2 w_{16} + 7 w_{25} \right)   \\
&- {\scriptstyle \frac{1}{3}} a_2 ( 2w_7 \gamma -  w_{20} ) \\
\Psi(e^3)_{36} &=  {\scriptstyle \frac{1}{12}} ( -8 \epsilon w_7 - 8 \gamma w_9 - 16 w_{15} + 20 w_{19} \\
&+ 9 w_{26} )  a_2 + {\scriptstyle \frac{1}{12}} \left( 8w_9 \gamma - 7 w_{26} \right) a_1 \\
\Psi(e^3)_{37} &=  {\scriptstyle \frac{1}{12}} a_1 \left( 4 \gamma w_{10} - 5 w_{27} - 2 w_{32} \right)   \\
&- {\scriptstyle \frac{1}{6}} a_2 \left( 2\gamma w_{10} + w_{27} - 2w_{32} \right) \\
\Psi(e^3)_{38} &= {\scriptstyle \frac{1}{3}} a_2 w_{21} - {\scriptstyle \frac{5}{12}} a_1 w_{28} 
\\
\Psi(e^3)_{39} &= {\scriptstyle \frac{1}{6}} a_1 w_{34}
\end{align*}
\end{multicols}

\begin{multicols}{2}
\begin{align*}
\Psi(e^3)_{40} &= {\scriptstyle \frac{1}{3}} a_2 w_{33} \\
\Psi(e^3)_{41} &= -{\scriptstyle \frac{1}{2}} a_2 \left( 4 \epsilon w_9 + 6 w_{16} - 9 w_{25} \right) \\
\Psi(e^3)_{42} &= {\scriptstyle \frac{1}{12}}   a_2\left( -8 \epsilon w_{10} - 9 w_{28} + 20 w_{31} \right) + {\scriptstyle \frac{1}{6}} a_1 w_{28} \\
\Psi(e^3)_{43} &= {\scriptstyle \frac{3}{4}} a_2 \left( w_{27} + 2 w_{32} \right) \\
\Psi(e^3)_{45} &= {\scriptstyle \frac{3}{2}} a_2 w_{34} \\
\Psi(e^3)_{46} &= {\scriptstyle \frac{1}{3}}  \left( a_1 - a_2 \right) w_{34} \\
\Psi(e^3)_{47} &= {\scriptstyle \frac{2}{3}} a_2 w_{33} \\
& \\
\Psi(e^4)_{11} &= -{\scriptstyle\frac{1}{12}}(a_{1} - 2a_{2})w_{7} \\
\Psi(e^4)_{12} &= {\scriptstyle\frac{1}{12}}(2a_{1} - a_{2})w_{9} 
\end{align*}
\begin{align*}  
\\ 
\Psi(e^4)_{13} &= {\scriptstyle\frac{1}{6}}(a_{1} + a_{2})w_{10}\\
\Psi(e^4)_{17} &= {\scriptstyle\frac{1}{4}}a_{1}w_{10} \\
\Psi(e^4)_{24} &= {\scriptstyle\frac{1}{4}}a_{2}w_{10} \\
\Psi(e^4)_{36} &= {\scriptstyle\frac{1}{12}}a_{1}(4\gamma w_{10} - 2w_{27} - 5w_{32}) \\
&- {\scriptstyle\frac{1}{12}}a_{2}(4\gamma w_{10} + 2w_{27} - 7w_{32}) 
\\
\Psi(e^4)_{37} &= -{\scriptstyle\frac{1}{12}}  (7a_{1} + 2a_{2})w_{33} \\
\Psi(e^4)_{38} &= -{\scriptstyle\frac{7}{12}}a_{1}w_{34} \\
\Psi(e^4)_{41} &= {\scriptstyle\frac{1}{6}}a_{2}(-4\epsilon w_{10} + 7w_{31}) - {\scriptstyle\frac{1}{3}}a_{1}w_{28} \\
\Psi(e^4)_{42} &= -{\scriptstyle\frac{1}{12}}(2a_{1} + 7a_{2})w_{34} \\
\Psi(e^4)_{43} &= {\scriptstyle\frac{7}{12}}a_{2}w_{33} 
\end{align*}
\end{multicols}

Let us now write down the results of the third step from Section \ref{ts-cec}, starting with the generic solutions. Moreover, this allows us to show that, generically, most of the Killing tensors are reducible.

\begin{prop}\label{generic-KT}
If  $a_1\neq a_2$,  then, generically, 
\begin{enumerate}
\item there is $22$-dimensional space of conformal Killing tensors, which are all reducible for $\epsilon\neq 0$, or 
\item there are irreducible Killing tensors from Proposition \ref{expl-KT} for $\epsilon= 0$.
\end{enumerate}
\end{prop}
\begin{proof}
The generic $\Phi$-invariant subspace $\mathbb{S}^{gen}$ annihilated by the curvature $R^{\Phi}$ of $\Phi$ can be computed directly; see Maple file \cite{bgg-comp}. The resulting subset $\mathbb{S}^{gen}\subset \mathbb{V}$ provides a $22$-dimensional space of Killing tensors. We have the following $22$ free variables
\begin{gather*} w_{1},\dots,w_{14},w_{17},\dots,w_{19},w_{24},w_{25},w_{35},w_{36},w_{41}
\end{gather*}
and the following dependent, where missing variables are either vanishing or are determined by the equations for $\mathbb{V}:$

\begin{multicols}{2}
\begin{align*}
w_{15}&=\epsilon w_7-w_{19}\\
w_{16}&=\epsilon w_9-w_{25}\\
w_{20}&=\gamma w_7\\
w_{26}&=\gamma w_9\\ 
w_{27}&= {\scriptstyle \frac13 } \gamma w_{10}\\
w_{31}&=\epsilon w_{10}\\
w_{32}&={\scriptstyle \frac23} \gamma w_{10}\\
w_{37}&=-2\epsilon^2w_7+{\scriptstyle \frac{1}{12}}a_1w_7-{\scriptstyle \frac16 }a_2w_7+3\epsilon w_{19}\\
w_{38}&=-2\epsilon\gamma w_7-{\scriptstyle \frac14 }a_1w_9+3\gamma w_{19}
\end{align*}
\begin{align*}
\\
w_{39}&=-{\scriptstyle \frac14 }aw_{10}\\
w_{40}&=-{\scriptstyle \frac13 }\gamma\epsilon w_{10}\\
w_{42}&=-2\epsilon^2w_9-{\scriptstyle \frac16 }a_1w_9+{\scriptstyle \frac{1}{12}}a_2w_9+3\epsilon w_{25}\\
w_{43}&=-2\epsilon\gamma w_9+{\scriptstyle \frac14 }a_2w_7+3\gamma w_{25}\\
w_{45}&=-{\scriptstyle \frac14 }aw_{10}\\
w_{46}&=-{\scriptstyle \frac16 }(-6\epsilon^2+a_1+a_2)w_{10}\\
w_{47}&={\scriptstyle \frac23}\gamma\epsilon w_{10}\\
w_{50}&=\gamma^2 w_{10}
\end{align*}
\end{multicols}

Comparing this with Proposition \ref{reducible-prop}, we see that these are reducible for dimensional and continuity reasons for $\epsilon\neq 0$. Otherwise, there are irreducible Killing tensors from Proposition \ref{expl-KT} for $\epsilon= 0$.
\end{proof}

In Proposition \ref{generic-KT}, we exhausted all the reducible Killing tensors. Therefore, the additional Killing tensors that may appear for special values of the parameters are irreducible. To finish the proof of Theorem \ref{main-KT}, we restrict ourselves to complement $C$ of $\mathbb{S}^{gen}$ in $\mathbb{V}$ (again parametrized by $v_i$ variables) and given by
\begin{gather*}
v_{1}=0,\dots,v_{14}=0,v_{17}=0,\dots,v_{19}=0,v_{24}=0,v_{25}=0,v_{35}=0,v_{36}=0,v_{41}=0.
\end{gather*}
We first find the subspace of $C$ annihilated by the curvature $R^{\Phi}$ of the prolongation connection $\Phi$ restricted to $C$, i.e., the subspace $C_0\subset C$ such that $R^{\Phi}(C_0)\subset \mathbb{S}^{gen}$. Here we immediately distinguish several branches.

\begin{lem}\label{curv-KT}
If  $a_1\neq a_2, \gamma\neq 0$,  then the subspace $C_0$ of $C$ such that $R^{\Phi}(C_0)\subset \mathbb{S}^{gen}$ is given by

\begin{align*}
v_{20} &={\scriptstyle  \frac{1}{\gamma(a_1-a_2)}}(-5\epsilon v_{15}-4v_{37})a_1\\
 v_{26} &= {\scriptstyle \frac{1}{\gamma(a_1-a_2)}}(5\epsilon v_{16}+4v_{42}) a_2\\
 v_{27} &= {\scriptstyle \frac{1}{\gamma(a1-a2)^2}} (-a_1^2\gamma v_{32}-a_2 (\gamma v_{32}+(2 v_{21}-2v_{28})\epsilon-4v_{39}+4v_{45})a_1+2a_2^2\gamma v_{32})\\
  v_{38} &={\scriptstyle  \frac14 \frac{1}{\gamma(a_1-a_2)} }\big(-5v_{15}(a_1-a_2)\gamma^2+(5(\epsilon v_{16}+ {\scriptstyle\frac45} v_{42}))(a_1+a_2)\gamma-(10(\epsilon v_{15}+ {\scriptstyle\frac45} v_{37}))\epsilon a_1\big)\\
 v_{40} &={\scriptstyle \frac{1}{12} \frac{1}{\gamma(a_1-a_2)^2}}\big((a_1-a_2)((v_{21}-5v_{28}+7v_{31})a_1-(7(v_{21}+{\scriptstyle \frac17 }v_{28}-2v_{31}))a_2)\gamma^2 \\ &-18v_{32}(a_1-a_2)(a_1+{\scriptstyle \frac12} a_2)\epsilon\gamma-12\epsilon ((v_{21}-v_{28})\epsilon-2v_{39}+2v_{45})a_2a_1\big) 
 \\
  v_{43} &= {\scriptstyle \frac14 \frac{1}{\gamma(a_1-a_2)}} \big(-5v_{16}(a_1-a_2)\gamma^2+(5(\epsilon v_{15}+ {\scriptstyle \frac45} v_{37}))(a_1+a_2)\gamma+10a_2(\epsilon v_{16}+ {\scriptstyle\frac45} v_{42})\epsilon \big)\\
 v_{46} &={\scriptstyle \frac{1}{4(a_1-a_2)} }\big(((-2v_{21}+7v_{31})\epsilon-3\gamma v_{32}+4v_{39})a_1+2a_2((v_{28}-{\scriptstyle \frac72} v_{31})\epsilon+{\scriptstyle \frac32} \gamma v_{32}-2v_{45})\big)\\
 v_{47} &= ({\scriptstyle \frac{1}{12}}(-2v_{21}-2v_{28}+7v_{31}))\gamma+{\scriptstyle \frac34} \epsilon v_{32}\\
 v_{50} &={\scriptstyle \frac{1}{2(a_1-a_2)} }\big((\epsilon v_{28}+3\gamma v_{32}-2 v_{45})a_1-3a_2(\gamma v_{32}+{\scriptstyle \frac13} v_{21}\epsilon-{\scriptstyle \frac23 }v_{39})\big)\\
\end{align*}
together with the distinguished condition $$(a_1a_2\gamma+a_2^2\gamma)v_{21}+(a_1^2\gamma+a_1a_2\gamma)v_{28}+3a_1a_2\epsilon v_{32}-7a_1a_2\gamma v_{31} = 0,$$
and the following equal to zero
\begin{gather*}
    v_{33}, v_{34}, v_{48}, v_{49}, v_{51 }, v_{52 }, v_{53 }, v_{54 }, v_{55 }.
\end{gather*}

If  $a_1\neq a_2, \gamma= 0$,  then the subspace $C_0$ of $C$ such that $R^{\Phi}(C_0)\subset \mathbb{S}^{gen}$ is given by

\begin{align*}
 v_{38} &= {\scriptstyle \frac{1}{4(a_1-a_2)}}\big(((5 v_{16}+2 v_{20})\epsilon+4 v_{42})a_1-2 a_2\epsilon v_{20}\big) \\
 v_{43} & = {\scriptstyle \frac{1}{4(a_1-a_2)}} \big(((5 v_{15}-2v_{26})\epsilon+4v_{37})a_2+2a_1\epsilon v_{26}\big) \\
  v_{46} &= {\scriptstyle  \frac{1}{4(a_1-a_2)}} \big(((-2 v_{21}+7v_{31})a_1+(2(v_{28}-\frac72 v_{31}))a_2)\epsilon+4a_1v_{39}-4a_{2}v_{45}\big) \\
   v_{47} &= {\scriptstyle \frac{1}{8(a_1-a_2)}} \big(((-4v_{27}+14v_{32})a_1-2a_2(v_{27}+{\scriptstyle \frac52} v_{32}))\epsilon+8 v_{40}(a_1+{\scriptstyle \frac12} a_2)\big)\\ 
    v_{48} &= {\scriptstyle \frac{1}{16(a_1-a_2)}} (13a_1-16a_2)\epsilon v_{33} \\
      v_{49} &={\scriptstyle \frac{1}{16(a_1-a_2)}} (16a_1-13a_2)\epsilon v_{34} \\
       v_{50} &= {\scriptstyle  \frac{1}{2(a_1-a_2)} }\big((\epsilon v_{28}-2 v_{45})a_1-a_2(\epsilon v_{21}-2v_{39})\big)  \\
  v_{51} &= -{\scriptstyle \frac{1}{64} \frac{1}{(a1-a2)^2}}a_1\epsilon v_{34}(16a_1-19a_2) \\
  v_{52} &= -{\scriptstyle \frac{1}{64} \frac{1}{(a1-a2)^2} } a_2 \epsilon  v_{33}(19a_1-16a_2) 
\end{align*}
together with the distinguished conditions
\begin{multicols}{2}
\begin{align*}
&\epsilon a_2 v_{34}=0 \\
&\epsilon a_1 v_{33}=0 \\
&\epsilon(4a_1^2 - 2a_2^2)v_{27} + \epsilon(-8a_1^2 + 4a_1a_2 + a_2^2)v_{32} \\
&+ (-8a_1^2 + 4a_2^2)v_{40} = 0 
\end{align*}
\begin{align*}\\
&a_1(5\epsilon v_{15} + 4v_{37}) = 0 \\
&a_2(5\epsilon v_{16} + 4v_{42}) = 0 \\
&a_1a_2(2\epsilon v_{27} - 3\epsilon v_{32} - 4v_{40}) = 0 \\
&a_1a_2(\epsilon v_{21} - \epsilon v_{28} - 2v_{39} + 2v_{45}) = 0
\end{align*}
\end{multicols}
and the following equal to zero
\begin{gather*}
    v_{53}, v_{54} ,v_{55}.
\end{gather*}
\end{lem}

We first deal with the branch $\epsilon=0,\gamma=0,a_1=0,a_2=\pm 2$, where we can directly find an additional irreducible Killing tensor from Proposition \ref{expl-KT}.

Next, we deal with the branch $\epsilon=0,\gamma=0,a_1\neq 0$, where we can assume $a_2=\pm 2+a_1\neq 0$. The distinguished equations for $C_0$ from Lemma \ref{curv-KT} simplify to $$v_{37}=0,v_{42}=0,v_{39}=0,v_{40}=0.$$ Computing the $\Phi$-invariant subspace of $C_0+\mathbb{S}^{gen}$ then leads to two distinguished cases $a_1=\pm \frac83,a_2=\pm \frac23$ with one additional irreducible Killing tensor from Proposition \ref{expl-KT}. Otherwise, there is no additional solution.

Next, we deal with the branch $\epsilon=1,\gamma=0,a_1= 0,a_2\neq 0$. The distinguished equations for $C_0$ from Lemma \ref{curv-KT} simplify to $$v_{16} = -\frac45 v_{42}, v_{32} = 2 v_{27}-4 v_{40}, v_{34} = 0.$$ 
Computing the $\Phi$-invariant subspace of $C_0+\mathbb{S}^{gen}$ then leads to several distinguished cases: 
\begin{itemize}
\item[(i)] Case $a_2=\frac34$ with the additional five-dimensional space of the irreducible Killing tensors of Proposition \ref{expl-KT}.
\item[(ii)] Case $a_2=-\frac3{16}$ with the additional six-dimensional space of the irreducible Killing tensors of Proposition \ref{expl-KT}. 
\item[(iii)] Otherwise, there is always an additional one-dimensional space of the irreducible Killing tensors of Proposition \ref{expl-KT}.
\end{itemize}

\medskip
Next, we deal with the branch $\epsilon=1,\gamma=0,a_1\neq 0,a_2\neq 0$. The distinguished equations for $C_0$ from Lemma \ref{curv-KT} simplify to
$$v_{15} = -\frac45 v_{37}, v_{21} = -2 v_{45}+v_{28}+2 v_{39}, v_{27} = 2 v_{40}, v_{32} = 0, v_{33} = 0, v_{34} = 0,v_{42}=-\frac54 v_{16}.$$ Computing the $\Phi$-invariant subspace of $C_0+\mathbb{S}^{gen}$ then leads to distinguished case $a_1=4a_2+\frac34$ with one additional irreducible Killing tensor from Proposition \ref{expl-KT}. Otherwise, there is no additional irreducible Killing tensor.

\medskip
Next, we deal with the branch $\epsilon=0,\gamma=1,a_1=0,a_2\neq 0$. The distinguished equations for $C_0$ from Lemma \ref{curv-KT} simplify to
$$v_{21}=0.$$ Computing the $\Phi$-invariant subspace of $C_0+\mathbb{S}^{gen}$ then shows that there is no additional irreducible Killing tensor.

\medskip
Next, we deal with the branch $\epsilon=0,\gamma=1,a_1\neq 0,a_2\neq 0$. The distinguished equations for $C_0$ from Lemma \ref{curv-KT} simplify to
$$v_{31} = \frac17 (\frac{a_1 a_2+a_2^2}{a_1a_2}v_{21}+\frac{a_1^2+a_1a_2}{a_1a_2}v_{28}).$$ Computing the $\Phi$-invariant subspace of $C_0+\mathbb{S}^{gen}$ then shows that there is no additional irreducible Killing tensor.

\medskip
Next, we deal with the branch $\epsilon=1,a_1=0,a_2\neq 0,\gamma\neq 0$. The distinguished equations for $C_0$ from Lemma \ref{curv-KT} simplify to
$$v_{21}=0.$$ Computing the $\Phi$-invariant subspace of $C_0+\mathbb{S}^{gen}$ then shows that there is no additional irreducible Killing tensor.

\medskip
Finally, we investigate the branch $\epsilon=1,a_1\neq 0,a_2\neq 0,\gamma\neq 0$. The distinguished equations for $C_0$ from Lemma \ref{curv-KT} allow us to eliminate
$v_{31}.$ The actual computations are done in \cite{bgg-comp}; we only briefly summarize here our computation of the $\Phi$-invariant subspace of $C_0+\mathbb{S}^{gen}$, because this goes beyond the direct solving of the parametric equations done in the previous cases. It requires us to generate equations for $\Phi(\fk)(C_0+\mathbb{S}^{gen})\subset C_0+\mathbb{S}^{gen}$, $\Phi(\fk)\Phi(\fk)(C_0+\mathbb{S}^{gen})\subset C_0+\mathbb{S}^{gen}$, $\Phi(\fk)\Phi(\fk)\Phi(\fk)(C_0+\mathbb{S}^{gen})\subset C_0+\mathbb{S}^{gen}$ and $\Phi(\fk)\Phi(\fk)\Phi(\fk)\Phi(\fk)(C_0+\mathbb{S}^{gen})\subset C_0+\mathbb{S}^{gen}$. These can be separated to equations for variables $v_{15},v_{16},v_{37},v_{42}$ and equations for variables $v_{21},v_{28},v_{32},v_{39},v_{45}$, and we separately show that there are no nontrivial solutions. We cannot directly solve the equations involving the parameters and thus, we solve them in the following alternative way:

We pick a pair of variables and eliminate the other variables from the equations. The linear dependence of the remaining equations leads to determinants built from the coefficients, and using Groebner bases, we show that there are no solutions except the algebraic subvarieties, where we divide by zero  during eliminating the other variables. After we restrict to the intersections of these algebraic subvarieties, we can show that there are no nontrivial solutions on them.

\end{document}